%% file: deg_final.tex
\documentclass[11pt]{amsart}
\usepackage{geometry}
\usepackage[english]{babel}
\usepackage[latin1]{inputenc}
\usepackage{amsfonts, amsmath, amssymb, amsthm, mathtools}
\usepackage{multicol}
\usepackage{paralist}
\usepackage{faktor}
\usepackage{color}

\usepackage{graphicx}
\usepackage{hyperref}
\usepackage{cleveref}
\crefname{thm}{Theorem}{Theorems}
\crefname{prop}{Proposition}{Propositions}
\crefname{lemma}{Lemma}{Lemmas}
\crefname{cor}{Corollary}{Corollaries}
\crefname{example}{Example}{Examples}
\crefname{defn}{Definition}{Definitions}
\crefname{rem}{Remark}{Remarks}

\theoremstyle{plain} 
\newtheorem{theorem}{Theorem}
\newtheorem{lemma}{Lemma}
\newtheorem{corollary}{Corollary}

\theoremstyle{definition}
\newtheorem{definition}{Definition}
\newtheorem{remark}{Remark}
\newtheorem{example}{Example}
\newtheorem{conjecture}{Conjecture}

\DeclareMathOperator{\N}{\mathbb N}
\DeclareMathOperator{\R}{\mathbb R}
\DeclareMathOperator{\C}{\mathbb C}

\DeclareMathOperator{\re}{Re}
\DeclareMathOperator{\im}{Im}

\DeclareMathOperator{\Span}{span}

\DeclareMathOperator{\Aut}{Aut}

\DeclareMathOperator{\tdeg}{td}

\newcommand{\Sphere}[1]{\mathbb{S}^{#1}}
\newcommand{\Heisenberg}[1]{\mathbb H^{#1}}

\newcommand{\sedit}[2]{#1}

\newcommand{\dopt}[1]{\frac{\partial}{\partial #1}}
\begin{document}

\title[On highly degenerate maps of spheres]{On highly degenerate CR maps of spheres}

\author{Giuseppe della Sala}
\address{Department of Mathematics, American University of Beirut}
\email{gd16@aub.edu.lb}
\author{Bernhard Lamel}
\address{Faculty of Mathematics, University of Vienna, Oskar-Morgenstern-Platz 1, 1090 Vienna, Austria}
\email{bernhard.lamel@univie.ac.at}
\author{Michael Reiter}
\address{Faculty of Mathematics, University of Vienna, Oskar-Morgenstern-Platz 1, 1090 Vienna, Austria}
\email{m.reiter@univie.ac.at}
\author{Duong Ngoc Son}
\address{Faculty of Fundamental Sciences, PHENIKAA University, Hanoi 12116, Vietnam}
\email{son.duongngoc@phenikaa-uni.edu.vn}

\begin{abstract}
For $N \geq 4$ we classify the $(N-3)$-degenerate smooth CR maps of the three-dimensional unit sphere into the $(2N-1)$-dimensional unit sphere. Each of these maps has image being contained in a five-dimensional complex-linear space and is of degree at most two, \sedit{or equivalent to one of the four maps into the five-dimensional sphere classified by Faran.}{when we exclude those maps which are contained in a complex three-dimensional subspace} As a byproduct of our classification we obtain new examples of rational maps of degree three which are $(N-3)$-degenerate only along a proper real subvariety and are not equivalent to polynomial maps. In particular, by changing the base point, it is possible to construct new families of nondegenerate maps.
\end{abstract}

\date{\today}
\maketitle


\section{Introduction}

The sphere plays an extremely important role in CR geometry as it is the homogeneous ``flat model'' for 
strictly pseudoconvex hypersurfaces in the sense of the canonical Cartan geometry 
developed by Tanaka \cite{Tanaka62} and Chern and Moser \cite{CM74}. We will denote the sphere in
$\C^n$ by 
\begin{align*}
\Sphere{2n-1} = \bigl\{z=(z_1,\ldots, z_n) \in \C^n: \|z\|^2 = |z_1|^2 + \cdots + |z_n|^2 = 1\bigr\},
\end{align*}
and its representation as the Heisenberg group by 
\begin{align*}
\Heisenberg{2n-1} \coloneqq \left\{ (z,w)=(z_1, \dots, z_{n-1},w) \in \C^{n-1} \times \C: \im w = \|z\|^2 \right\}.
\end{align*}
There is one more special role for $ \Sphere{2n-1}$, which is that in higher codimension $m\geq n$,
there is an abundance of holomorphic
maps taking $\Sphere{2n-1}$ into $\Sphere{2m-1}$, and yet, the
space of these holomorphic maps carries a rich structure not yet well understood, 
and we will discuss some of the results in
order to put our results in context.


We will call holomorphic maps taking $\Sphere{2n-1}$ into $\Sphere{2m-1}$ \emph{sphere maps} from now on. Any sphere map can be composed with automorphisms of the spheres to obtain a new sphere map. In order to exclude these trivial \sedit{constructions}{maps}, we consider the class of sphere maps up to equivalence with respect to the composition of automorphisms of the source sphere and the target spheres. \\
Two notions of equivalence are important here: The first one with respect
to the isotropy groups at a distinguished point, and the second one with respect to the full groups, which include {\em translations}, i.e., automorphisms belonging to the transitive part of the groups with respect to a distinguished point. In order to distinguish the two, we are going to refer to the first one as {\em isotropical equivalence}, and the second one simply as {\em equivalence}.

We note that we might as well talk about smooth CR maps between spheres, and work by Forstneri\v{c} \cite{Forstneric89} and Cima--Suffridge \cite{CS90} shows that such maps are rational with poles outside of the sphere.

It turns out that in certain cases that (excluding constant maps) the linear embedding $L\colon \C^n \to \C^m$ given by $L(z) = (z,0)$ is the only sphere map up to equivalence. This phenomenon is referred to as \emph{rigidity} and was observed by Alexander \cite{Alexander74} in the case $m=n \geq 2$, Webster \cite{Webster79b,Webster79a} for $m=n+1 \geq 3$, and finally Faran \cite{Faran86}  and Huang \cite{Huang99} when $3\leq n \leq m \leq 2m-2$.  These conditions are typically 
thought of as ensuring ``low codimension'' in a suitable sense. 

The situation is wildly different in ``high codimension''. 
When $n\geq 3$ and $m= 2n-1$ Huang--Ji \cite{HJ01} found that there are two equivalence classes of sphere maps: The linear embedding $L$ and the \emph{Whitney map} $W$ given by
\begin{align*}
W(z_1, \ldots, z_n) = (z_1, \ldots, z_{n-1}, z_1z_n, z_2z_n, \ldots, z_n^2).
\end{align*}
When $m=2 n$, there exists the so-called ``D'Angelo'' family. These maps are of the form
\begin{align*}
F_s(z_1,\ldots, z_n) = \bigl(z_1, \ldots, z_{n-1}, \sqrt s z_n, \sqrt{1-s}  z_1 z_n,\sqrt{1-s}  z_2 z_n, \ldots, \sqrt{1-s} z_n^2  \bigr),
\end{align*}
and are pairwise inequivalent for $0\leq s \leq 1$, so that there are infinitely many equivalence classes. This phenomenon persists for $m\geq 2n$, see D'Angelo \cite{DAngelo88a}; 
Hamada \cite{Hamada05} showed that for $m=2n$ (and $n\geq 4$) these exhaust all equivalence classes. Note that $F_0$ is equivalent to the map $W \oplus 0$, while $F_1 = L$.  
It was discovered by Huang--Ji--Xu \cite{HJX06} that for $n\geq 4$ and $n \leq m \leq 3 n-4$ there are no ``new maps'' apart from $F_s$ composed with the linear embedding $L$, i.e., maps equivalent to $F_s \oplus 0$, where $0$ is the zero vector of appropriate length. This observation is now know as the \emph{gap phenomenon}, see the survey \cite{HJY08}. There are further instances given in \cite{HJY14} and \cite{AHJY16}. 

D'Angelo's result and the gap phenomenon show that for $n=2$ rigidity fails for every $m\geq 3$, 
we obtain ``new maps'' at every step, and we 
have to think about every $m\geq n$ as a ``high codimension'' case.  Faran \cite{Faran82} proved that there are four equivalence classes of sphere maps when $n=2$ and $m=3$, given by 
\[ (z_1, z_2) \mapsto \begin{cases}
\left( z_1^2 , \sqrt{2} z_1 z_2 , z_2^2 \right) \\
\left( z_1, z_2 z_1, z_2 z_1^2 \right) \\
\left( z_1^3, \sqrt{3} z_1 z_2 , z_2^3 \right) \\
\left( z_1,z_2,0 \right). 
\end{cases} \]
 Other proofs were provided by Cima--Suffridge \cite{CS89}, Ji \cite{Ji10} and the third author \cite{Reiter16a}. The case where the real dimension of the source sphere is three has been examined
  by many authors from different angles: D'Angelo \cite{DAngelo88b}, \cite{DAngelo91}, D'Angelo--Kos--Riehl \cite{DKR03}, D'Angelo--Lebl \cite{DL09}, \cite{DL11}, \cite{DL16}, D'Angelo--Grundmeier--Lebl \cite{DGL20}, D'Angelo--Xiao \cite{DX17}, Lichtblau \cite{Lichtblau92} and Meylan \cite{Meylan06}.

Let us highlight the following structural results: 
Lebl has shown in \cite{Lebl11} that any rational sphere map of degree two from $\Sphere{3}$ into $\Sphere{2m-1}$ for $m\geq 4$ 
belongs to the 
family of monomial maps $D^m_{s,t}: \Sphere{3} \rightarrow \Sphere{2m-1}$, given by
\begin{align}
\label{eq:familyMap}
D^m_{s,t}(z_1,z_2) = \bigl(\sqrt{s} z_1, \sqrt{t} z_2, \sqrt{1-s} z_1^2, \sqrt{2 - s -t} z_1 z_2, \sqrt{1-t} z_2^2, 0, \ldots, 0\bigr),
\end{align}
where $s$ and $t$ are two real parameters satisfying $0\leq s\leq t \leq 1$ 
and $(s,t) \neq (1,1)$. Two maps of this family are equivalent if and only if they have the same parameters.
When $m=4$, then necessarily either $s=0$ or $t=1$ and we obtain the two one-parameter families of maps listed in \cite{Watanabe92}. This classification generalizes previously obtained results in \cite{CJX06, HJX06} and \cite{JZ09}.

The complexity of sphere maps from $\Sphere{3}$ into $\Sphere{2m-1}$ is an intriguing, but widely open problem. When $m=4$ there exists the classification of monomial maps by D'Angelo \cite{DAngelo88a}, see also \cite{Watanabe92}.
So far, a complete list of sphere maps from $\Sphere{3}$ into $\Sphere{7}$ remains open. 

Our main motivation in this paper is to investigate a particular 
case of sphere mappings from $\Sphere{3}$ in complete detail. 
We call these maps {\em tangentially degenerate} at $p$. Roughly this means that their derivatives are 
linearly dependent along the complex tangent plane at $p$   and we 
 refer the reader to \Cref{def:tangdeg} for the technical definition.


\begin{theorem}
\label{thm:deg3families}
If $H: (\Heisenberg{3},0) \rightarrow (\Heisenberg{2N-1},0)$ is a smooth CR map which is tangentially $(2,N-3)$-degenerate at $0$,
then $H$ is either quadratic or isotropically equivalent to one of the following families of maps:
\begin{compactenum}[\rm i)]
\item \begin{align*}
H = (f, z^2 \phi, w(a_3 z + b_3 w) \phi, b_4 w^2 \phi, 0,  \dots, 0, g)/\delta,
\end{align*}
where 
\begin{align*}
 f & = z-\frac{i}{2}(2 \mu -1) w z+\frac{\lambda}{2}w^2+\lambda ^2 w^2 z+i \lambda  w z^2, \\
 \phi & = 1+2 \lambda  z, \\
 g & = w (1-i \mu  w),\\
 \delta & = 1-i \mu  w-i \lambda  w z-w^2 \left(2 \lambda ^2+\frac{\mu }{2}\right)-2 \lambda  \mu  w^2 z+2 i \lambda ^2 \mu  w^3,
\end{align*}
satisfying the following conditions:
\begin{align*}
\mu > 1/6, \qquad 0 < \lambda \leq \sqrt{\mu(6\mu-1)/3},
\end{align*}
and $a_3 = \sqrt{6 \mu-1}/2$ and $b_3 = i \lambda/(4 a_3)$. Additionally if $N\geq 5$, then $b_4 = (2 \mu^2 - \lambda^2)/4- |b_3|^2$, and if $N=4$, then $\lambda = \sqrt{\mu(6 \mu -1)/3}$.
\item \begin{align*}
H = (f, (z^2 + a w^2)\phi, a_3 w (z - i\lambda w) \phi, 0, \ldots, 0, g)/\delta,
\end{align*}
where 
 \begin{align*}
 f & = 2z+i (4 a+1) w z +\lambda  w^2 -8 a z^3+2 i (8 a+1) \lambda  w z^2 \\
     & -w^2 z \left(2 a^2-8 a \lambda ^2-a-2 \lambda ^2\right) -i a \lambda  w^3,\\
 \phi & = 2 (1+2 \lambda  z-i a w),\\
 g & = 2w\left(1+2 i a w-4 a z^2+8 i a \lambda  w z+ aw^2 \left(4\lambda ^2-a\right)\right),\\
 \delta & = 2+4 i a w-8 a z^2+2 i  \lambda  (8 a-1) w z-2 w^2 \left(a^2-4 a \lambda ^2-a+2 \lambda ^2\right) \notag \\
    &~ -4 i a w z^2-6 a \lambda  w^2 z+2 i a^2 w^3,
 \end{align*}
 and $\lambda\geq 0, a\leq -1/16$ and $a_3 = \sqrt{-(1+16a)}/2$. If $N=3$, then
 $a=-1/16$.
\end{compactenum}
\end{theorem}

\begin{remark}
\label{rem:reduceOneParam}
We observe that, when restricting the maps in (ii) to $a =-1/16$, the image of $\Heisenberg{3}$ projects to $\Heisenberg{5}$ and we obtain, after setting $\lambda=s/2$ the one-parameter family of maps of degree $3$ from \cite[Theorem 4.1]{Reiter16a}.
\end{remark}

The proof of \Cref{thm:deg3families} is provided in \cref{sec:proofMainThm}. 

It turns out that the two-parameter families of maps obtained as isotropical 
equivalence classes in \Cref{thm:deg3families} are not equivalent with respect 
to (small) translations.

\begin{corollary}\label{cor:localrigid}
Any map listed in \Cref{thm:deg3families} is locally rigid, i.e., isolated in the quotient space of $(2,N-3)$-tangentially degenerate maps with respect to the full group of automorphisms. In other words, there is at most a discrete set of translations, which transform one given map into other maps of the list. 
\end{corollary}

See also \cite{dSLR15b}, \cite{dSLR15a}, \cite{dSLR17} and \cite{dSLR18} for more information on local rigidity of CR maps.

The previous \Cref{thm:deg3families} allows to solve one part of the classification of sphere maps and we prove the following result:

\begin{theorem}
\label{thm:constDegMaps}
Let $H: \Sphere{3} \rightarrow \Sphere{2N-1}$ be a smooth CR map. Assume that the image of $\Sphere{3}$ under $H$ is not contained in a complex three-dimensional subspace of $\C^N$. Then $H$ is constantly $(N-3)$-degenerate in an open set of $\Sphere{3}$ if and only if $H$ is quadratic.
\end{theorem}

A sphere map $H: \Sphere{3} \rightarrow \Sphere{2N-1}$ is \emph{constantly $s$-degenerate} in the sense of \cite{Lamel01} if, at least generically, the space of holomorphic vector fields tangent to $\Sphere{2N-1}$ along the image of $H$ is of dimension $s$. When $s=0$ the map is called \emph{finitely nondegenerate}. See \cite{Reiter19} and \cref{sec:degeneracy} for more details and an equivalent definition of degeneracy.

The motivation for our work was inspired by the analogue to \Cref{thm:constDegMaps} in higher dimensions, which was obtained in \cite{CJY19}, when the source dimension $n$ is greater than $4$.

It is well-known that if the map is $(N-1)$-degenerate on an open set, then it must be a constant map, see \cref{sec:N-1-Deg} below. If the map is $(N-2)$-degenerate on an open set, then it must be a linear map, see \cref{sec:N-2-Deg} below. The degree of the maps stated in \Cref{thm:constDegMaps} does not follow from general considerations or known degree bounds. Together with \Cref{thm:constDegMaps} we want to propose the following conjecture:

\begin{conjecture}
If a sphere map is $s$-degenerate for $s> 0$ having linearly independent components, then it is of degree $N-s-1$.
\end{conjecture}

Note that when $s=0$, the above conjecture cannot hold: The \emph{group-invariant sphere map} from $\Sphere{3}$ into $\Sphere{2N-1}$ is finitely nondegenerate and of degree $2N-3$, see \cite{DAngelo88b} and \cite{Reiter19}.

We would also like to point out that the target dimension of the maps given in \Cref{thm:constDegMaps} cannot be deduced from the known bounds on the size of the image of sphere maps given in \cite{CJY19} (see also \cref{sec:SizeImage} below), which shows that these bounds are not sharp.

The proof of \Cref{thm:constDegMaps} follows the line of thought carried out in \cite{Reiter16a}: We first provide a (partial) normal form for $(N-3)$-degenerate sphere maps with respect to automorphisms, which belong to the stabilizer of a base point. \\
Then we classify the normalized sphere maps, by using the fact that they are $(N-3)$-degenerate. Initially, this allows to reduce the complexity of the map by showing that all but two components of the map are a multiple of one holomorphic function by a linear or quadratic factor, depending on the case (Case I or II in \cref{sec:proofMainThm}) under consideration as stated in \Cref{lem:lindepmap} below. A lot of computations aim to determine the map up to its $4$-jet. Again by the $(N-3)$-degeneracy the map must satisfy a system of three holomorphic functional equations, which allow us to solve for the components of the map. \\
We end up with two two-parameter families of rational maps of degree three as given in \Cref{thm:deg3families}, which contain all normalized sphere maps, when we exclude those maps which are quadratic or which embed into $\Sphere{5}$. 
An interesting fact appearing in the proof of \Cref{thm:deg3families} is, that case II, when most of the components factor by a linear term, actually is a subcase of case I, when most of the components factor by a quadratic factor. This follows from tedious computations: It would be interesting to be able to conclude this directly from the degeneracy condition. \\
The obtained maps in \Cref{thm:deg3families} turn out to be $(N-3)$-degenerate at most along a proper real subvariety of $\Sphere{3}$, hence all constantly $(N-3)$-degenerate sphere maps are quadratic or embed into $\Sphere{5}$.

Let us close this introduction by giving an overview of each section. 
In \cref{sec:prelim} we provide all relevant definitions, in particular of the definition of degeneracy and give a bound on the size of the image of a sphere map, depending on the degeneracy of the map. In \cref{sec:normalform} we obtain a partial normal form for the $(N-3)$-degenerate sphere maps with respect to automorphisms fixing a point. In \cref{sec:initial} we use the degeneracy to obtain an initial form of the map and show that all but two components are a multiple of a holomorphic function and a linear or quadratic factor. 
In \cref{sec:proofMainThm} we give the proof of our main result \Cref{thm:deg3families}. Let us point out that, that we used Mathematica in \cref{sec:proofMainThm}. Nevertheless, all computations are elementary and straight-forward, but often tedious and contain expressions with many terms involved. We tried to  find the shortest path to achieve our goal, although often a high level of complexity is inevitable, but we hope our representation meets this aspiration. Finally, in \cref{sec:propmaps} we deduce some properties of the maps and the proofs of \Cref{cor:localrigid} and \Cref{thm:constDegMaps}.

\section{Preliminaries}
\label{sec:prelim}

In this section we provide relevant notions and definitions, and introduce the notation that we
are going to use. 

\subsection{The unit sphere and its automorphisms}
For  coordinates in $\C^2$ and $\C^N$, $N\geq 2$, respectively,  and their conjugates as well as 
coordinates in the complexification we write $Z=(z,w), \bar Z =(\chi,\tau)$  and  \[ \begin{aligned}
Z'&=(z',w')=(z_1',\ldots,z_{N-1}',w'), \\ 
\bar Z'&=(\chi',\tau')=(\chi_1',\ldots,\chi_{N-1}',\tau').
\end{aligned} \]
For vectors $a=(a_1, \ldots, a_m) \in \C^m$ and $b=(b_1, \ldots, b_m) \in \C^m$ we write $\langle a, b \rangle =  a_1 b_1 + \cdots + a_m b_m$ and $\|a\|^2 = \langle a, \bar a \rangle$.


The \emph{unit sphere} $\Sphere{2N-1}$ in $\C^N$ is given by
\begin{align*}
\Sphere{2N-1} \coloneqq \left\{z \in \C^N: \|z\|^2 = 1 \right\}. 
\end{align*} 
Another representation of the sphere is given by the \emph{Heisenberg hypersurface} $\Heisenberg{2N-1}$ in $\C^N$, defined by
\begin{align*}
\Heisenberg{2N-1} \coloneqq \left\{ (z,w) \in \C^{N-1} \times \C: \im w = \|z\|^2 \right\}.
\end{align*}
$\Sphere{2N-1} \setminus \left\{ (0,-1) \right\}$ is biholomorphically equivalent to 
$\Heisenberg{2N-1}$ via the Cayley transform 
\[\Sphere{2N-1} \setminus \left\{ (0,1) \right\} \ni (z_1, \dots ,z_{N-1}, z_N) \mapsto \left( \frac{z_1}{1+z_N} ,\dots , 
\frac{z_{N-1}}{1+z_N}, i \frac{1-z_N}{1+z_N} \right) \in \Heisenberg{2N-1},  \]
with inverse 
\[\Heisenberg{2N-1} \ni \left( z_1, \dots, z_{N-1},w \right) \mapsto \left( \frac{2i z_1}{i+w}, \dots, \frac{2i z_{N-1}}{i+w}, \frac{i-w}{i+w} \right).   \]

We are going to work with the CR vector field on $\Heisenberg{3}$ given by 
$L=\frac{\partial}{\partial \chi} - 2 i z \frac{\partial}{\partial \tau}$.
For a smooth function $h$ depending on two complex variables $z$ and $w$ we write $h^{(k,\ell)}$ for $\frac{\partial^{k+\ell} h(0)}{\partial z^k \partial w^\ell}$.

We denote by $\Aut(\Heisenberg{2N-1})$ the automorphisms of $\Heisenberg{2N-1}$, that is, the linear fractional holomorphic maps taking $\Heisenberg{2N-1}$
into itself. Those have at most one pole on $\Heisenberg{2N-1}$.  
We also denote by $\Aut(\Heisenberg{2N-1},0)$ the collection of those 
automorphisms which are  germs of biholomorphisms at $0$  of $\C^N$ mapping $\Heisenberg{2N-1}$ into itself. 
The automorphisms of $\Heisenberg{2N-1}$ which fix $0$ form a group, which we denote by $\Aut_0(\Heisenberg{2N-1},0)$ and refer to as the \emph{stability group} of $(\Heisenberg{2N-1},0)$. Elements of the stability group can be parametrized as follows:
\begin{align*}
\phi_{\lambda,U,c,r} (z,w) = \frac{\left(\lambda U (z + c w), \lambda^2 w \right)}{1 - 2 i \langle \bar c, z \rangle + (r - i \|c\|^2) w},
\end{align*}
where $\lambda > 0$, $U$ is a unitary matrix of size $N-1$, $c \in \C^{N-1}$ and $r \in \R$. In addition, we are going to use the following  automorphisms 
of $\Heisenberg{2N-1}$ which send $0$ to 
$p_0 =(z_0, u_0 + i \|z_0\|^2) \in \Heisenberg{2N-1}$: 
\begin{align*}
t_{p_0}(z,w) = \bigl(z+ z_0, w + u_0 + i \|z_0\|^2 + 2 i  \langle \bar z_0, z \rangle \bigr).
\end{align*}
The collection of all such automorphisms of $\Heisenberg{2N-1}$ together with their inverses we refer to as (proper) \emph{translations}; they actually can be thought of as elements  of the stability group of the 
point at infinity. We shall also refer to the map $t_\infty \colon (z,w) \mapsto ( \frac{z}{w}, - \frac{1}{w})$ sending the point $0$ to the point 
at infinity as a ``translation".  The proper translations form a subgroup of $\Aut(\Heisenberg{2N-1})$, which will be denoted by $T(\Heisenberg{2N-1})$.

Equivalently, we can work with the sphere $\Sphere{2N-1}$; its automorphisms are again the 
linear fractional maps of $\C^N$ taking $\Sphere{2N-1}$ into itself. While in the unbounded 
representation, a map can well have a pole on $\Heisenberg{2N-1}$ (corresponding to sending a 
point to the point at infinity), the pole set in the compact representation is outside of the 
sphere.

\subsection{CR maps and spherical equivalence}
Let $H$ be a smooth CR map from 
$\Heisenberg{3}$ into $\Heisenberg{2N-1}$ written as 
 $H=(f,g)=(f_1,\ldots, f_{N-1},g)$. 
We recall again that works of Forstneri\v{c} \cite{Forstneric89} and Cima--Suffridge \cite{CS90} show that these maps are (restrictions of) rational holomorphic maps with poles outside of $\Heisenberg{3}$. Meylan \cite{Meylan06} showed that their  degree is at most $N(N-1)/2$,  where the \emph{degree} of a rational map is the maximal degree of the polynomials in the numerators and denominators of each component of the map. Degree bounds
for sphere maps are an interesting topic in themselves, and 
we refer the reader to D'Angelo's recent book \cite{DAngeloBook21} for 
more information on them. 

The map $H$ satisfies the \emph{mapping  equation} (also 
referred to as the {\em basic equation}), which in its ``complexified"
form is given by
\begin{align*}
g(z,\tau + 2 i z \chi) - \bar g(\chi,\tau) = 2 i \langle f(z, \tau + 2 i z \chi), \bar f(\chi,\tau) \rangle.
\end{align*}
This equation holds for all $z,w,\chi,\tau \in \C$ for which the
expression is defined (we have to avoid the poles, of course). 
An equivalent notion for the mapping equation is as follows: $H$ sends $\Heisenberg{3}$ into $\Heisenberg{2N-1}$ if there is a 
rational function $Q$  with no pole at $0$ such that 
\begin{align*}
\frac{g(z,w) - \bar g(\chi,\tau)}{2 i} - \langle f(z, w), \bar f(\chi,\tau) \rangle = Q(z,w,\chi, \tau) \left(\frac{w -\tau}{2i} - z \chi \right).
\end{align*}

Composing a map with automorphisms trivially generates new maps. One is therefore interested 
in the equivalence classes of maps with respect to ``spherical equivalence".

\begin{definition}
Let $H_1,H_2 \colon \C^n \to \C^N $ be sphere 
maps. We say that $H_1$ and $H_2$ are {\em equivalent} if there exist 
automorphisms $\psi \in \Aut (\Heisenberg{2n-1})$ and $\psi' \in \Aut (\Heisenberg{2N-1})$ such that $H_2 = \psi' \circ H_1 \circ \psi^{-1}$. 

Equivalently, this means that  there are open, connected subsets $U_1, U_2 \subset \Heisenberg{2n-1}$ on which $H_1$ and $H_2$ are actually 
defined,  if there are translations $t_p \in T(\Heisenberg{2n-1})$ for $p \in \Heisenberg{2n-1}$ and $t_q' \in T(\Heisenberg{2N-1})$ for $q \in \Heisenberg{2N-1}$, $\phi \in \Aut_0(\Heisenberg{3},0)$ and $\phi' \in \Aut_0(\Heisenberg{2N-1},0)$ such that
\begin{align*}
H_2 = \phi' \circ t'_q \circ H_1 \circ t_p \circ \phi.
\end{align*} 
\end{definition}

\begin{remark}
\label{rem:transversality}
By a Hopf lemma argument which seems to be ``folklore" (see e.g. \cite[section 2]{Huang99} and \cite[Lemma 3.1]{ES12b} for  proofs), it follows
 that a nonconstant CR map $H=(f,g)$ of 
 $\Heisenberg{2n-1}$ into $\Heisenberg{2N-1}$ satisfies $g^{(0,1)}>0$, which means that $H$ is \emph{CR transversal} at all points. Now the mapping 
 equation implies $g^{(0,1)} = \langle f^{(1,0)}, {\bar f}^{(1,0)}\rangle$, so $H$ is 
 immersive. 
\end{remark}

\subsection{Degeneracy of a CR map}
\label{sec:degeneracy}
A useful biholomorphic invariant in the study of CR mappings in positive codimension is their degeneracy as introduced in \cite{Lamel01}. We first 
recall the general definition: 

\begin{definition}
\label{def:finiteDeg}
Let $M \subset \C^N$ and $M'\subset \C^{N'}$ be real-analytic generic CR submanifolds of codimension $d$ and $d'$ respectively. Close to $p\in M$ and $p'\in M'$ we assume that $M$ and $M'$ are given by the vanishing of $\rho$ and $\rho'$ respectively. 
Write $L_1, \ldots, L_n$ for a basis of CR vector fields of $M$. 
Consider a smooth CR map $H: M \rightarrow M'$. For each $k\in \N$ define the following subspaces of $\C^{N'}$:
\begin{align*}
E_k'(p) \coloneqq \Span_{\C} \{L^\alpha {\rho'_{j z'}}(H(z),\overline{H(z)})|_{z=p}: \alpha \in \N^n, |\alpha| \leq k,1\leq j \leq d'\}.
\end{align*}
The number $s(p) \coloneqq N' - \max_k \dim_{\C} E_k'(p)$ is referred to as \emph{degeneracy} of $H$ at $p$. The map $H$ is of \emph{constant degeneracy} $s(q)$ at $q\in M$ if $p\mapsto s(p)$ is a constant function in a neighborhood of $q$. 
For the smallest integer $k_0$ such that $E_\ell'(p) = E_{k_0}'(p)$ for all $\ell \geq k_0$ we say the map $H$ is \emph{$(k_0,s)$-degenerate} at $p$. When $s=0$ we say the map is \emph{finitely nondegenerate} at $p$ or more precisely, that it is \emph{$k_0$-nondegenerate} at $p$.
\end{definition}

Note that for a map $H= (f,g)$ from $\Heisenberg{3}$ into 
$\Heisenberg{2N-1}$ we have 
\[ E_k (0) = \Span_{\C} \left\{ (0,1), (f^{(1,0)},0), \dots , (f^{(k,0)},0) \right\}.\]

\begin{remark}
\label{rem:estimateDeg}
It follows directly from the definition of degeneracy and the transitivity of the automorphism group that any map can be assumed to be of constant degeneracy $0\leq s \leq N-1$ in a neighborhood of the origin in $\Heisenberg{3}$ and sending $0$ to $0$, when considering maps up to equivalence.\\
\end{remark}

\subsection{Tangential degeneracy} 
\label{sub:tangential_degeneracy}
In the following, we will be concerned with maps degenerating along 
specific subsets; depending on how these are situated with respect 
to the CR vector fields, the degeneration along these sets has strong
implications. For this, we work directly with 
the Heisenberg hypersurface, denote its complexification by 
$\mathcal{M}$, with defining equation
equal to $w - \tau = 2 i z \chi$. The CR field $L$ extends naturally 
to a vector field on the complexification, again denoted by $L$, if 
we set 
\[ L =\dopt{\chi} - 2 i z \dopt{\tau}, \]
and consider a map $H=(f,g)$ as before. 
We note that for $ (p,\bar q) = (z,w,\chi,\tau) \in \mathcal{M}$, we can also 
define
\[  {\tilde E}_\ell (p, \bar q) =  {\rm span } \left\{ (L^k \bar f) (p,\bar q) \colon 
k \leq \ell \right\} \subset \C^{N-1}, \quad
{\tilde E} (p,\bar q) = \bigcup_{\ell} {\tilde E}_\ell (p, \bar q) . \] 
The \emph{degeneracy} of $H$ at $(p,\bar q)$ is then defined to be $d(p,\bar q) = N-1-\dim \tilde E(p,q)$, 
and the \emph{generic degeneracy} is given by  $d_H = \max_{(p,\bar q) \in \mathcal{M}} d(p,\bar q)$, realized outside of a complex-analytic 
subvariety of $\mathcal{ M}$.
Note that generically (outside of a complex analytic subvariety of $\mathcal{M}$), 
$\tilde E (p, \bar q) = \tilde E_{N-1-d_H}(p,\bar q)$. 

We also recall that with the Segre variety $S_q$ associated to $q = (\bar \chi,\bar \tau)$, 
which in our case is 
the complex hyperplane given by $w= \tau + 2 i z \chi $, we have $(p,\bar q) \in \mathcal{M}$
if and only if $p \in S_q$. We now define the tangential degeneracy of a sphere map $H$. 

\begin{definition}\label{def:tangDeg} The \emph{tangential degeneracy} of $H$ at $q$ is defined to be 
\[ \tdeg_H (q) = \max_{p \in S_q} d (p,\bar q).  \]
We say that $H$ is \emph{tangentially $(k,\tdeg_H (q))$-degenerate} if $\tilde E_k (p,\bar q) = \tilde E(p,\bar q)$ for generic $p\in S_q)$ and $k$ is the smallest integer with that property. 
  \label{def:tangdeg}
  \end{definition}  


\subsection{$(N-1)$-degenerate maps} 
\label{sec:N-1-Deg}
We first recall that the maps of highest degeneracy are necessarily constant. We don't yet need the 
concept of tangential degeneracy for that. 

\begin{lemma}
Let $H=(f,g)=(f_1,\ldots, f_{N-1},g): \Heisenberg{2n-1} \rightarrow \Heisenberg{2N-1}$ be a map. 
If $H$ is of degeneracy $N-1$ at any point, then $H$ is constant. 
\end{lemma}
\begin{proof}
Constant degeneracy $N-1$ means in particular that $L \bar f = 0$ in $\Heisenberg{3}$. If $H$ is nonconstant, then it is 
transversal at $0$,
which implies that $\bar f^{(1,0)} = L\bar f(0) \neq 0$ (see \Cref{rem:transversality}). Hence $H$ must be constant.
\end{proof}

\subsection{Tangentially $(1,N-2)$-degenerate maps}
\label{sec:N-2-Deg}

We note that a similar characterization appears in \cite[Theorem 1.3]{CJY19} using different results and methods, under a slightly stronger assumption than here. 

\begin{lemma}
\label{lem:N-2-Deg}
Let $H=(f,g)=(f_1,\ldots, f_{N-1},g): \Heisenberg{3} \rightarrow \Heisenberg{2N-1}$ be a map with 
$\tdeg_H(q) = N-1$ for some $q\in \Heisenberg{3}$. Then $H$ is equivalent to a linear embedding. 
\end{lemma}

\begin{proof}
After translating $q$ to $0$, we partially normalize the map. Since $H$ is transversal, we may assume that $f^{(1,0)}=e_1$ after applying a unitary matrix to $f$. Next, we consider $\Phi' \circ H$, where 
$\Phi'$ is the automorphism of $\Heisenberg{2N-1}$ defined by
\begin{align*}
\Phi'(z',w') \coloneqq \frac{\left(z'+ c' w' ,w'\right)}{1-2 i \langle z',\bar c' \rangle - i \|c'\|^2 w'}, \qquad c'  \coloneqq - f^{(0,1)} \in \C^{N-1},
\end{align*}
so that we may assume $f^{(0,1)}=0$. 
From the degeneracy assumption we have
\begin{align}
\label{eq:det2}
\left|\begin{array}{cc}
L \bar f_1 & L \bar f_k\\
L^m \bar f_1 & L^m \bar f_k
\end{array}\right| (z,0,0,0) = 0, 
\end{align}
for $1\leq k \leq N-1, m\geq 1$. We follow \cite[Lemma 16]{Lamel01} and use the fact that for a power series $\phi \in \C[[z,w,\chi,\tau]]$ we have 
\begin{align}
\label{eq:CRPowerSeries}
& \phi(z,w,\chi,w-2 i z \chi) = \phi(z,w,0,w)+\sum_{k\geq 1} \frac{1}{k!}\frac{\partial^k}{\partial \chi^k}\biggl|_{\chi=0} \phi(z,w,\chi,w-2i z \chi) \chi^k \\
\nonumber
& = \phi(z,w,0,w)+\sum_{k\geq 1} \frac{1}{k!} (L^k \phi) (z,w,0,w) \chi^k.
\end{align}


Using \eqref{eq:det2} we obtain 
\begin{align*}
(L \bar f_1)(z,0,0,0) (L^m \bar f_k)(z,0,0,0) - (L \bar f_k)(z,0,0,0) (L^m \bar f_1)(z,0,0,0)=0.
\end{align*}
We divide by $m!$, multiply with $\chi^m$, sum over $m\geq 0$ and use the above fact \eqref{eq:CRPowerSeries}, to get
\begin{align*}
L\bar f_1(z,0,0,0) &(\bar f_k(\chi,-2 i z \chi)-\bar f_k(0)) \\ 
&- L \bar f_k(z,0,0,0) (\bar f_1(\chi,-2 i z \chi)-\bar f_1(0))=0, 
\end{align*}
i.e.
\begin{align*}
 \left(\bar f_{1}^{(1,0)} - 2 i z \bar f_{1}^{(0,1)}\right) \bar f_k(\chi,-2 i z \chi) -  \left(\bar f_{k}^{(1,0)} - 2 i z \bar f_{k}^{(0,1)}\right) \bar f_1(\chi,-2 i z \chi)  = 0.
\end{align*}

Using the normalization conditions $f^{(1,0)}=e_1$ and $f^{(0,1)}=0$ we get $f_k(z,2 i z \chi) = 0$ 
for $2\leq k \leq N-1$. This implies that $f_k = 0$ for $2\leq k \leq N-1$ and
so  $H = (f_1,0,g)$ gives rise to the map $(f_1,g)$ from $\Heisenberg{3}$ into itself, which is a linear automorphisms by a well-known result of Alexander \cite{Alexander74}.
\end{proof}

\subsection{Tangentially $(2,N-3)$-degenerate maps}
In the next section, we are going to consider
 maps $H=(f_1,\ldots, f_{N-1},g): \Heisenberg{3} \rightarrow \Heisenberg{2N-1}$, which are tangentially $(2,N-3)$-degenerate at the origin. In the 
spirit of the preceding section, this means that  
we can assume  
\begin{align*}
\left|\begin{array}{cc}
L \bar f_1(0) & L \bar f_2(0)\\
L^2 \bar f_1(0) &  L^2 \bar f_2(0)
\end{array}\right| = \left|\begin{array}{cc}
f_{1}^{(1,0)} & f_{2}^{(1,0)}\\
f_{1}^{(2,0)} & f_{2}^{(2,0)}
\end{array}\right| \neq 0,
\end{align*}
and
\begin{align}
\label{eq:det1}
\left|\begin{array}{ccc}
L \bar f_1 & L \bar f_2 & L \bar f_k\\
L^2 \bar f_1 & L^2 \bar f_2 & L^2 \bar f_k\\
L^m \bar f_1 & L^m \bar f_2 & L^m \bar f_k
\end{array}\right| (z,0,0,0) = 0,
\end{align}
for $1\leq k \leq N-1$ and $m\geq 1$. 
In particular, these maps include all of the maps which are 
constantly $(N-3)$-degenerate in a neighborhood of the origin. 

\begin{example}
From the list of monomial maps in \cite{LS20}, the only maps of constant degeneracy $1$ from $\Heisenberg{3}$ to $\Heisenberg{7}$ are those, which project to a map from $\Heisenberg{3}$ into $\Heisenberg{5}$ and the two one-parameter families of quadratic maps. The quadratic map $D^N_{s,t}: \Sphere{3} \rightarrow \Sphere{2N-1}$ with parameters $s,t \in \R$ in \eqref{eq:familyMap} is constantly $(N-3)$-degenerate.
\end{example}

\subsection{Bounds on the size of image}
\label{sec:SizeImage}
In this section we study degenerate maps in general. The following estimate on the size of the image of a degenerate map was given before in \cite{CJY19} (see also \cite[section 5, step (II)]{HJY14}, which itself is based on \cite{Huang99}, proof of Lemma 4.3). For the sake of completeness we provide a proof here.

\begin{lemma}
\label{lem:dimbound}
Let $H$ be a nonconstant map, which is constantly $(k_0,s)$-degenerate in a neighborhood of the origin and sends $\Heisenberg{3} \subset \C^2$ into $\Heisenberg{2N-1} \subset \C^N$. Then
\begin{itemize}
\item[(a)] $N=s+k_0+1$
\item[(b)] The image of $H$ is contained in a subspace of $\C^N$ of dimension at most  $\binom{k_0+2}{2}$.
\end{itemize}
\end{lemma}

\begin{remark}
The estimate in \cite{CJY19} differs from the one obtained in \cref{lem:dimbound}. By the normal form in \cite[Theorem 2.1]{CJY19} when the geometric rank $\kappa_0 \leq n-2$, one has the property that pure $w$-derivatives (of order bigger than $1$) at $0$ vanish. In fact, using the notation of \cite{CJY19}, it is assumed that $n\geq 5$ and $\kappa_0 \leq 2$. Hence, in \cite[(4.2)]{CJY19} the derivatives when $|\alpha| = 0$ (assuming $1\leq b \leq l_0-1$) can be ignored, which results in the bound $\xi(l_0) \leq \binom{l_0}{2}$. Thus it is estimated that the image of the map is contained in a subspace of dimension $d_{l_0}+2+\xi(l_0) \leq l_0-1 + 2 + \binom{l_0}{2} = \binom{l_0+1}{2}+1< \binom{l_0+2}{2}$.
\end{remark}

\begin{proof}[Proof of \cref{lem:dimbound}]
To show (a), we write $H=(f,g) \in \C^{N-1} \times \C$ and
\begin{align*}
v_0=(\bar f,i/2) \in \C^{N-1} \times \C. 
\end{align*}
and define for $1\leq m \leq k$,
\begin{align*}
v_m=(L^m \bar f,0) \in \C^{N-1} \times \C,
\end{align*}
and for $k \geq 1$ we define the subspaces 
\begin{align*}
E_k \coloneqq \Span_{\C} \{v_0,v_1, \ldots, v_k\} \subseteq \C^N.
\end{align*}
By \cite[Lemma 23--24]{Lamel01} we have the generic bounds $0\leq s\leq N-2$ and $1\leq k_0 \leq N-1-s$, which imply that the vectors $\{v_0,v_1,\ldots, v_{k_0}\}$, spanning $E_{k_0}$, (generically) are linearly independent. Indeed, $E_{k_0}$ must span an $(N-s)$-dimensional subspace, which is only achieved if $k_0 +1 = N-s$.

For (b) we consider the tangential component $f$ of $H=(f,g)$ and define the subspace
\begin{align*}
F_{k_0}(p) \coloneqq \Span_{\C} \{L^a \bar f(p): a\leq k_0\} \subseteq \C^{N-1}.
\end{align*}
Observe, that by the degeneracy, it holds that $L^b \bar f (p) \in F_{k_0}(p)$ for all $b\in \N$. Write $T=\partial/\partial w + \partial/\partial \bar w$ such that $T = \frac{1}{2i} [ L, \bar L]$ commutes with $L$ and $\bar L$ and $T \bar f= 1/(2 i) \bar L(L \bar f)$.
For $k \geq 0$ we define the subspaces
\begin{align*}
D_k(p) \coloneqq \Span_{\C}\{T^a L^b \bar f(p): a+b \leq k\} \subseteq \C^{N-1}.
\end{align*}

Let $(c,d) \in \N^2$. Then we have, since $L^{c+d} \bar f \in F_{k_0}$ and $\bar L^j L^k \bar f = 0$ if $j > k$, that 
\begin{align*}
T^c L^d \bar f = \underbrace{\frac 1 {(2i)^c} \bar L^c L^{c+d} \bar f}_{\in \bar L^c F_{k_0}} = \sum_{\substack{j\leq \min(c,k_0) \\ k\leq k_0}} e_{jk} \bar L^{j} L^{k} \bar f \in D_{k_0}, \qquad e_{jk}\in \C.
\end{align*}
This shows $D_k \subseteq D_{k_0}$ for all $k\geq 0$ and thus
\begin{align*}
f(Z) = \sum_\alpha \frac 1{\alpha!} \frac{\partial^\alpha f (0)}{\partial Z^\alpha}Z^\alpha \in \Span_{\C} \left\{\frac{\partial^\alpha f (0)}{\partial Z^\alpha}: \alpha \in \N^2\right\} = \bigcup_{k \geq 0} {D_k(0)} = { D_{k_0}(0)}. 
\end{align*}

This implies that the image of $H$ is contained in $D_{k_0}(0)$. 
To compute the latter, we note that $D_{k_0}(0)$ is spanned by the derivatives of $H$ w.r.t. $z$ and $w$ at $0$ up to order $k_0$. Thus $\dim_{\C}D_{k_0}(0)$ is at most the dimension of the space of polynomials in two variables of degree at most $k_0$, which is $\binom{k_0+2}{2}$.
\end{proof}

\section{Partial normal form}
\label{sec:normalform}

In this section we deduce a partial normal form for $(N-3)$-degenerate maps. The following is inspired by \cite{Huang99, Huang03, HJ01, Ji10}. The normalization conditions agree with those given in \cite{Reiter16a} when $N=3$.

\begin{lemma}
\label{lem:normalform}
Let $H=(f,g)=(f_1,\ldots,f_{N-1},g): \Heisenberg{3} \rightarrow \Heisenberg{2N-1}$ be a map with $L H (0)$ and $ L^2 H (0)$ linearly independent. Then by composing with automorphisms of $\Heisenberg{3}$ and $\Heisenberg{2N-1}$ fixing $0$ we can assume the following conditions:
\begin{multicols}{2}
\begin{compactenum}[\rm i)]
\item $H^{(1,0)}=e_1$
\item $H^{(0,1)}=e_N$
\item $H^{(2,0)}=2 e_2$
\item $f_2^{(1,1)}=0$
\item $\re\left(g^{(0,2)}\right)=0$
\item $\re\left(f_2^{(2,1)}\right)=0$
\item $f_1^{(0,2)}\geq 0$
\end{compactenum}
\end{multicols}
If a map $H$ satisfies the 
normalization conditions (i)-(vii), it also satisfies
\begin{multicols}{2}
\begin{compactenum}[\rm i)]\setcounter{enumi}{7}
\item  $f_1(z,0)=z$
\item $g_w(z,0) = 1$
\item $\im\left(g^{(0,2)}\right) = 0$
\item $f_1^{(1,1)} = \frac i 2$
\end{compactenum}
\end{multicols}
If furthermore we assume that $H$ is tangentially $(2,N-3)$-degenerate, then we also have 
\begin{compactenum}[\rm i)]\setcounter{enumi}{11}
\item $f_k^{(m,0)} = 0$ for $k\geq 3$ and $m\geq 1$.
\end{compactenum}
\end{lemma}

\begin{proof}

We start with a map $H_{[0]} \coloneqq H$, which we know is transversal.  Note, that from the mapping equation it follows that $g(z,0) = 0$ and $g^{(0,1)} = \|f^{(1,0)}\|^2$.  These conditions and $g^{(0,1)}>0$ are preserved when the map is composed with automorphisms. Furthermore, by assumption, 
$f^{(1,0)}$ and $f^{(2,0)}$ are linearly independent; after reordering, we assume that 
\begin{equation}\label{e:indepass}\begin{vmatrix}
	f_1^{(1,0)} & f_2^{(1,0)} \\ 
	f_1^{(2,0)} & f_2^{(2,0)}
\end{vmatrix}  \neq 0.\end{equation}
The proof now proceeds by choosing a sequence of automorphisms guaranteeing that 
the normalization conditions are satisfied.

Our notation for the gradual application of automorphisms is as follows: We write $H_{[k]} \coloneqq \phi'_{[k]} \circ H_{[k-1]} \circ \phi_{[k]}$, for $k \geq 1$, where $\phi_{[k]}$ and $\phi'_{[k]}$ are automorphisms of $\Heisenberg{3}$ and $\Heisenberg{2N-1}$ respectively. 
The components we denote as $H_{[k]} = (f_{[k]},g_{[k]})=(f_{[k],1},\ldots, f_{[k],N-1}, g_{[k]})$. When dealing with components of $H_{[0]}$ and its derivatives we skip the index $[0]$.

To simplify the formulas, we write $v_{\alpha}  \coloneqq (v_{\alpha_1},v_{\alpha_2}, \ldots, v_{\alpha_r}) \in \C^r$, for $v\in \C^N$ and $\alpha =(\alpha_1,\ldots, \alpha_r) \in \N^r$. Moreover, $e_j$ denotes the $j$-th standard unit vector in $\C^m$, where $m\in \{N-1,N\}$, depending on the context and $I_k$ is the $(k\times k)$-identity matrix.

\subsubsection*{Normalizing the $(1,0)$-coefficients of $f$}
We apply an $(N-1)\times (N-1)$-unitary matrix to $f$, to obtain
\begin{align*}
f_{[1]}^{(1,0)} & = \left\|f^{(1,0)}\right\|e_1 \in \C^{N-1},
\end{align*}
where the transversality of $H$ is used.

\subsubsection*{Normalizing the $(2,0)$-coefficients of $f$}
We apply another $(N-2)\times (N-2)$-unitary matrix to $(f_{[1],2}, \ldots, f_{[1],N-1})$ to obtain 
\begin{align*}
f_{[2]}^{(2,0)} =\left(\ast, \left\|f^{(2,0)}_{[1],(2,\ldots, N-1)}\right\|, 0, \ldots, 0\right),
\end{align*}
where we have used the condition given in \eqref{e:indepass}.
Note that these normalization conditions and the degeneracy condition imply that $f_{[2],k}^{(m,0)} = 0$ for $k\geq 3$ and $m\geq 1$. These conditions are preserved in the following.

\subsubsection*{Scaling $f_1^{(1,0)}$ and $f_2^{(2,0)}$}
We define the following diagonal matrix
\begin{align*}
D'\coloneqq \left(\begin{array}{cccc}
\lambda' &  &  &  \\
 &  \ddots &   &  \\
 &  &  \lambda'  & \\
 & &     & \lambda'^2
\end{array} \right),
\end{align*}
where blank spaces are filled up with zeros, and consider $H_{[3]}(z,w) \coloneqq D' H_{[2]}(\lambda z, \lambda^2 w)$ with parameters
\begin{align*}
\lambda \coloneqq \frac{2 \left\|f^{(1,0)}\right\|}{\left\|f^{(2,0)}_{[1],(2,\ldots, N-1)}\right\|}, \qquad \lambda' \coloneqq \frac{\left\|f^{(2,0)}_{[1],(2,\ldots, N-1)}\right\|}{2 \left\|f^{(1,0)}\right\|^2 }.
\end{align*}
Then
\begin{align*}
f_{[3]}^{(1,0)} = e_1, \qquad f_{[3]}^{(2,0)} = (\ast, 2, 0, \ldots, 0).
\end{align*}

Consequently, assuming that $f_{[3]}^{(1,0)}=e_1$, it follows from the mapping equation that $g_{[3]}^{(0,1)} = 1$, which is preserved in the following steps.

\subsubsection*{Normalizing the $(0,1)$-coefficients of $f$} We define $H_{[4]}\coloneqq \phi'_{[4]} \circ H_{[3]}$, where 
\begin{align*}
\phi'_{[4]}(z',w') \coloneqq \frac{\left(z'+ c' w' ,w'\right)}{1-2 i \langle z',\bar c' \rangle - i \|c'\|^2 w'}, \qquad c'  \coloneqq -f_{[3]}^{(0,1)}.
\end{align*}
Then
\begin{align*}
f_{[4]}^{(1,0)}  = e_1,  \quad f_{[4]}^{(0,1)}  = 0, \quad f_{[4]}^{(2,0)}  = \left(\ast, 2, 0, \ldots, 0\right).
\end{align*}

Assuming $H_{[4]}^{(1,0)}=e_1$ and $H_{[4]}^{(0,1)}=e_N$, it follows from the mapping equation that $g_{[4]w}(z,0) = 1$, $f_{[4],1}(z,0) = z$ and $g_{[4]}^{(0,2)} \in \R$. These conditions are preserved in the following steps.

\subsubsection*{Normalizing $f^{(1,1)}_2$} We define $H_{[5]}\coloneqq \phi'_{[5]} \circ H_{[4]} \circ \phi_{[5]}$, where 
\begin{align*}
\phi_{[5]}(z,w) & \coloneqq \frac{\left(z+d w, w \right)}{1-2 i \bar d z - i |d|^2 w}, \qquad  \qquad d  \coloneqq - \frac{f_{[4],2}^{(1,1)}} 2,\\
\phi'_{[5]}(z',w') & \coloneqq \frac{\left(z'+ d' w' ,w'\right)}{1-2 i \langle z',\bar d' \rangle - i \|d'\|^2 w'}, \qquad d' \coloneqq - d e_1.
\end{align*}
Then
\begin{align*}
f_{[5]}^{(1,0)}  = e_1, \quad f_{[5]}^{(0,1)} = 0, \quad f_{[5]}^{(2,0)}  = 2 e_2, \quad f_{[5],2}^{(1,1)} = 0.
\end{align*}

\subsubsection*{Normalizing $f^{(0,2)}_1$} We define $H_{[6]}(z,w) \coloneqq U' H_{[5]}(u z, w)$, where
\begin{align*}
U' \coloneqq \left(\begin{array}{ccc}
u' & 0 & 0  \\
0 & v' & 0  \\
0 & 0 & I_{N-2} 
\end{array} \right),
\end{align*}
with parameters
\begin{align*}
u'  \coloneqq \frac{\bar f_{[5],1}^{(0,2)}}{\left|f_{[5],1}^{(0,2)}\right|}, \quad v'  \coloneqq u'^2,\quad u \coloneqq u'^{-1}.
\end{align*}

Then 
\begin{align}
f_{[6]}^{(1,0)} & = e_1, \quad f_{[6]}^{(0,1)} = 0, \quad f_{[6]}^{(2,0)}  = 2 e_2, \quad 
f_{[6],2}^{(1,1)} = 0, \quad f_{[6],1 }^{(0,2)}  = \left|f_{[5],1}^{(0,2)}\right|.
\end{align}

\subsubsection*{Normalizing $\re\left(f^{(2,1)}_2\right)$ and $\re\left(g^{(0,2)}\right)$} We define $H_{[7]} \coloneqq \phi'_{[7]} \circ H_{[6]} \circ \phi_{[7]}$, where
\begin{align*}
\phi_{[7]}(z,w) & \coloneqq \frac{(z,w)}{1+ r w}, \qquad r \coloneqq \frac{\re\left(f_{[6],2}^{(2,1)}\right) - g_{[6]}^{(0,2)}}{2}\\ 
\phi'_{[7]} (z',w') & \coloneqq \frac{(z',w')}{1+r' w'}, \qquad r' \coloneqq \frac{-\re\left(f_{[6],2}^{(2,1)}\right)+2 g_{[6]}^{(0,2)}}{2}.
\end{align*}
Then 
\begin{align*}
H_{[7]}^{(1,0)} & = e_1, \quad H_{[7]}^{(0,1)} = e_N, \quad H_{[7]}^{(2,0)}  = 2 e_2,  \quad 
f_{[7],2}^{(1,1)} = 0, \quad f_{[7],1}^{(0,2)} = \left|f_{[5],1}^{(0,2)}\right|, \\
f_{[7],2}^{(2,1)} &  = i \im\left(f_{[6],2}^{(2,1)} \right), \quad g_{[7]}^{(0,2)} =0.
\end{align*}
Hence $H_{[7]}$ satisfies all required normalization conditions. Assuming $H_{[7]}^{(1,0)} = e_1$, $H_{[7]}^{(2,0)} = 2 e_2$ and $g_{[7]}^{(0,2)} = 0$, it follows from the mapping equation that $f_{[7],1}^{(1,1)} = i / 2$.
\end{proof}

\section{Initial form and statement of classification}
\label{sec:initial}

In this section we study $(N-3)$-degenerate maps and reduce their complexity in a first step.

\begin{lemma}
\label{lem:lindep}
Let $H =(f,g)=(f_1,\ldots, f_{N-1},g): \Heisenberg{3} \rightarrow \Heisenberg{2N-1}$ be a map which is tangentially $(2,N-3)$-degenerate at the origin, and satisfying the partial normal form of \cref{lem:normalform}. Then, for $3\leq k \leq N-1$, the component $f_k$ satisfies
\begin{align}
\label{eq:lindep}
\left(2 z^2+ w^2 f_{2}^{(0,2)}\right) f_k(z,w) = w \left(2 z f_{k}^{(1,1)} + w f_{k}^{(0,2)}\right) f_2(z,w).
\end{align}
\end{lemma}

\begin{proof}
We proceed similarly as in the proof of \cref{lem:N-2-Deg} and consider the determinant in \eqref{eq:det1}, which is
\begin{align*}
\left|\left(\begin{array}{ccc}
L \bar f_1 & L \bar f_2 & L \bar f_k\\
L^2 \bar f_1 & L^2 \bar f_2 & L^2 \bar f_k\\
L^m \bar f_1 & L^m \bar f_2 & L^m \bar f_k
\end{array}\right)(z,0,0,0)\right| = 0.
\end{align*}

Divide this equation by $m!$, multiply with $\chi^m$, sum over $m\geq 0$ and use the fact from \eqref{eq:CRPowerSeries}, to get

\begin{align*}
\left|\begin{array}{ccc}
L \bar f_1(z,0,0,0) & L \bar f_2(z,0,0,0) & L \bar f_k(z,0,0,0)\\
L^2 \bar f_1(z,0,0,0) & L^2 \bar f_2(z,0,0,0) & L^2 \bar f_k(z,0,0,0)\\
\bar f_1(\chi,0-2 i z \chi)  & \bar f_2(\chi,-2 i z \chi)& \bar f_k(\chi,-2 i z \chi)
\end{array}\right| = 0.
\end{align*}

Setting $w=0$, assuming the normalization conditions $f^{(1,0)}=e_1, f^{(0,1)}=0$, such that $L \bar f(z,0,0,0) = e_1$, we obtain
\begin{align*}
\left|\begin{array}{cc}
L^2 \bar f_2(z,0,0,0) & L^2 \bar f_k(z,0,0,0)\\
\bar f_2(\chi,-2i z \chi) & \bar f_k(\chi,-2i z \chi)
\end{array}\right| \equiv 0.
\end{align*}


It follows that 
\begin{align*}
& \left(\bar f_{2}^{(2,0)} -4 iz \bar f_{2}^{(1,1)} - 4 z^2 \bar f_{2}^{(0,2)}\right) \bar f_k(\chi,-2i z \chi) \\ 
=  & \left(\bar f_{k}^{(2,0)} - 4 i z \bar f_{k}^{(1,1)} - 4 z^2 \bar f_{k}^{(0,2)} \right) \bar f_2(z,-2 i z \chi).
\end{align*}

By the normal form we have that $f^{(2,0)}=2 e_2, f_{2}^{(1,1)}=0$, such that:
\begin{align*}
\left(1-2 z^2 \bar f_{2}^{(0,2)}\right) \bar f_k(\chi,-2 i z \chi) = -2 z \left(i \bar f_{k}^{(1,1)} + z \bar f_{k}^{(0,2)}\right) \bar f_2(\chi,-2 i z \chi).
\end{align*}

Thus we obtain after setting $z=\tau/(-2 i \chi)$ and clearing the denominator,
\begin{align*}
\left(2 \chi^2+ \tau^2\bar f_{2}^{(0,2)}\right) \bar f_k(\chi,\tau) = \tau \left(2 \chi \bar f_{k}^{(1,1)} + \tau \bar f_{k}^{(0,2)}\right) \bar f_2(\chi,\tau),
\end{align*}
which gives the desired formula after conjugating.
\end{proof}

We write
\begin{align*}
\ell_k(z,w) \coloneqq f_{k}^{(1,1)} z  + \frac{1}{2}f_{k}^{(0,2)} w, \quad k\geq 2,  \quad 
\ell(z,w)\coloneqq \frac 1 {2 \left(f_{3}^{(1,1)}\right)^2}\left(2 f_{3}^{(1,1)} z-  f_{3}^{(0,2)} w\right),
\end{align*}
and
\begin{align*}
q_2(z,w) \coloneqq z^2 +  \frac{1}{2}f_{2}^{(0,2)}w^2, \qquad q_k(z,w) \coloneqq w \ell_k (z,w) = w \left(f_{k}^{(1,1)} z + \frac{1}{2}f_{k}^{(0,2)} w\right), \quad k \geq 3.
\end{align*}

Using this notation we can further specialize to maps of a certain form.

\begin{lemma}\label{lem:lindepmap}
Assume that a map $H: \Heisenberg{3} \rightarrow \Heisenberg{2N-1}$ satisfies \eqref{eq:lindep} and  the partial normal form of \cref{lem:normalform}. Then there exist 
holomorphic functions $f$, $\phi$ and $g$ as well as numbers 
$\mu_4, \dots, \mu_{N-1}$  such that 
\begin{compactenum}[\rm i)]
\item  $H=(f,\ell \phi, w \phi, \mu_4 w\phi ,\dots, \mu_{N-1} w \phi, g)$, where $\phi(0)=0, \phi^{(1,0)}=f_{3}^{(1,1)}, \phi^{(0,1)}=f_{3}^{(0,2)}/2$.
\item $H=(f,q_2 \phi, q_3 \phi, \ldots,q_{N-1} \phi,g)$ with $\phi$ satisfying $\phi(0)=1$ and $\phi^{(0,1)} \in i \R$.
\end{compactenum}
\end{lemma}

\begin{proof}
With the notations introduced above in \eqref{eq:lindep}, we have 
\[ q_2 (z,w) f_k (z,w) = w \ell_k (z,w) f_2 (z,w) \]
for every $k$. Thus, for every $k$, $\ell_k$ either divides $q_2$ or $f_k$. 

If there exists a $k\geq 3$ for which $\ell_k$ divides $f_k$, assume without 
loss of generality that $k=3$; writing $\ell_3 \tilde \phi = f_3$, 
we see that $q_2  \tilde \phi = w  f_2$, hence $\tilde \phi = w \phi$. We thus
have $f_2 = q_2 \phi$, and thus $f_k = q_k \phi$ for $k\geq 2$. 
Ensuring that the normalization conditions on $f_2$
are satisfied leads to the restrictions on $\phi$ in case (ii).

Next, assume that all of the $\ell_k$ divide $q_2$; in particular for $k=3$, we have $\ell_3 \ell = q_2$ and so 
$ \ell  f_3 = w  f_2$, from which we conclude  $f_2 = \ell h_1$ 
and $f_3 = w h_1$, for some $h_1$.  

We again distinguish two cases. If $\ell_3$ does not divide
$\ell_k$ for some $k > 3$, then $\ell_3 \ell f_k = w \ell_k \ell h_1$ 
implies $\ell_3$ divides $h_1$. So $h_1 = \ell_3 \phi$, and we see that 
those maps fall into category (ii) again. 

On the other hand, if $\ell_3$ divides $\ell_k$ for all $k$, then 
$\ell_k = \mu_k \ell_3$ for some $\mu_k \in \C$ and it follows that 
$f_k = \mu_k w h_1 = \mu_k f_3$; this gives the maps in (i).
\end{proof}

The maps of \cref{lem:lindepmap} can be further normalized as follows:

\begin{lemma}
\label{lem:initalPlus}
By applying unitary matrices the maps from \cref{lem:lindepmap} can be normalized further as follows:
\begin{compactenum}[\rm i)]
\item If $N=4$  and $H$ is of the form \cref{lem:lindepmap} (ii) we can additionally assume that either $f_{3}^{(1,1)}>0$ or $f_{3}^{(0,2)}>0$.  
\item If $H$ is of the form \cref{lem:lindepmap} (i) we can additionally assume that $f_{3}^{(1,1)}>0$ and that $\mu_k = 0$, $k\geq 4$.
\item If $N>4$ and  $H$ is of the form \cref{lem:lindepmap} (ii) we can additionally assume that $H$ is of the form
\begin{align*}
H=(f,q_2 \phi,\tilde q_3 \phi,\tilde q_4 \phi, 0 , \ldots, 0, g),
\end{align*}
where $\tilde q_k(z,w) = w(2 z \tilde a_k + w \tilde b_k)$ for $k=3,4$ with $\tilde a_3 > 0, \tilde b_3\in \C, \tilde a_4=0$ and $\tilde b_4>0$.
\end{compactenum}
\end{lemma}

\begin{proof}
To show (iii) we set $a\coloneqq \left(f_{3}^{(1,1)}, \ldots, f_{N-1}^{(1,1)}\right)\in \C^{N-3}$ and $b\coloneqq \left(f_{3}^{(0,2)}, \ldots, f_{N-1}^{(0,2)}\right) \in \C^{N-3}$, such that
\begin{align*}
H(z,w)=(f(z,w),q_2(z,w) \phi(z,w), (2 a z w + b w^2) \phi(z,w),g(z,w)).
\end{align*}
We apply an $N\times N$-matrix $U_1$ to $H$, which is of the form
\begin{align*}
U_1 \coloneqq \left(\begin{array}{ccc}
I_2 &  0 & 0\\
0  & V_1 & 0\\
0 & 0& 1
\end{array} \right),
\end{align*}
where $V_1$ is an $(N-3)\times (N-3)$-unitary matrix such that $V_1 a = (\tilde a_3, 0, \ldots, 0) \in \C^{N-3}$ and $\tilde a_3 > 0$. We apply another $N\times N$-matrix $U_2$ to $U_1 H$, which is of the form
\begin{align*}
U_2 \coloneqq \left(\begin{array}{ccc}
I_3 & 0 & 0\\
0 & V_2 & 0\\
0 &  0& 1
\end{array} \right),
\end{align*}
where $V_2$ is an $(N-4)\times (N-4)$-unitary matrix such that $V_2 b = (\tilde b_4, 0, \ldots, 0) \in \C^{N-4}$ and $\tilde b_4 > 0$.\\
We proceed similarly  in (ii) to delete the $\mu_j$. For the other claims, note that in (ii)  
we require $f_{3zw}(0)\neq 0$, while in (i) we require that not both $f_{3 zw}(0)$ and $f_{3 w^2}(0)$ are zero.
\end{proof}

\section{Proof of classification}
\label{sec:proofMainThm}
\input{theproof.tex}


\subsection{Properties of the maps}
\label{sec:propmaps}

In this section we discuss some properties of the maps obtained in \Cref{thm:deg3families}.\\
When $N=4$ we take $\mu = 1/3$ in (i) to obtain
\begin{align*}
H^1_4(z,w) &  =  (f_1,f_2,f_3,g)/\delta,
\end{align*}
where
\begin{align*}
f_1 & = 3 \left(w^2 (2 z+3)+3 i w z (2 z+1)+18 z\right),\\
f_2 & = 18 z^2 (2 z+3),\\
f_3 & = -3 w (2 z+3) (w-3 i z),\\
g & = 18 w (w+3 i),\\
\delta & = 54 + 4 i w^3-3 w^2 (4 z+7)-18 i w (z+1).
\end{align*}
To obtain the other map for $N=4$, we take $\lambda = 0$ and $a=-1$ in (ii) of \Cref{thm:deg3families} to obtain
\begin{align*}
H^2_4(z,w) & = (f_1,f_2,f_3,g)/\delta,
\end{align*}
where
\begin{align}
\label{eq:H24a}
f_1 & = i z \left(3 w^2+3 i w-8 z^2-2\right),\\
f_2 & = (i-w) \left(w^2-z^2\right),\\
f_3 & = \sqrt{15} (w-i) z w,\\
g & = i w \left(w^2+2 i w-4 z^2-1\right),\\
\label{eq:H24b}
\delta & = 2 \left(w^3+2 i w^2+2 w \left(z^2-1\right)-i \left(4 z^2+1\right)\right).
\end{align}

Both maps $H_4^1$ and $H_4^2$ are nondegenerate in an open, dense subset of $\Heisenberg{3}$ and are $1$-degenerate at $0$.

When $N\geq 5$ we take $\lambda = \mu = 1$ in (i) of \Cref{thm:deg3families} to obtain
\begin{align*}
H^1_N(z,w) & = (f_1,f_2,f_3,f_4, 0, \ldots, 0, g)/\delta,
\end{align*}
where
\begin{align*}
f_1 & = \sqrt{5} \left(w^2 (2 z+1)+i z w (2 z-1)+2 z\right),\\
f_2 & = 2 \sqrt{5} z^2 (2 z+1),\\
f_3 & = w (2 z+1) (w-5 i z),\\
f_4 & = 2 i w^2 (2 z+1),\\
g & = 2 \sqrt{5}
   w (w+i),\\
\delta & = \sqrt{5} \left(2i-4 w^3-i w^2 (4 z+5)+2 w (z+1)\right).
\end{align*}
To obtain the other map for $N\geq 5$ we take the same parameters as for $N=4$, i.e., $\lambda = 0$ and $a=-1$, in (ii) of \Cref{thm:deg3families} to obtain
\begin{align*}
H^2_N(z,w) & = (f_1,f_2,f_3,0,0,\ldots, 0, g)/\delta,
\end{align*}
where $f_1,f_2,f_3,g$ and $\delta$ are given in \eqref{eq:H24a}-\eqref{eq:H24b}.

The above maps $H_4^1, H_4^2, H_N^1$ and $H_N^2$ correspond to maps from $\Sphere{3}$ into $\Sphere{2N-1}$, denoted by $F_4^1,F_4^2, F_N^1$ and $F_N^2$ respectively: When $N=4$ we have
\begin{align*}
F^1_4(z,w) & = (f_1,f_2,f_3,g)/\delta,
\end{align*}
where
\begin{align*}
f_1 & = 6 \left(\left(19 w^2+40 w+13\right) z+6 (w-1) z^2-3 (w+1)(w-1)^2\right),\\
f_2 & = 36 z^2 (3 w+2 z+3),\\
f_3 & = -6 i (w-1) (3 w+2 z+3) (w+3z-1),\\
g & = 25+111 w-6 \left(w^2+4 w-5\right) z+207 w^2+89 w^3,\\
\delta & = 169 +183 w-6 \left(w^2+4 w-5\right) z+63 w^2+ 17 w^3.
\end{align*}
\begin{align*}
F^2_4(z,w) & = (f_1,f_2,f_3,g)/\delta,
\end{align*}
where
\begin{align}
\label{eq:F24a}
f_1 & = z \left(w^2-w+4 z^2+4\right),\\
f_2 & = 2 w \left((w-1)^2+z^2\right),\\
f_3 & = \sqrt{15} i(1-w) z w,\\
g & = (5 w-1) z^2+(w+1)^2,\\
\label{eq:F24b}
\delta & = 5-2 w + w^2+(w+3) z^2.
\end{align}
When $N\geq 5$ we have
\begin{align*}
F^1_N(z,w) & = (f_1,f_2,f_3,f_4,0,\ldots, 0, g)/\delta,
\end{align*}
where
\begin{align*}
f_1 & = 2 \sqrt{5} \left(\left(1+8w-w^2\right) z+2 (w-1) z^2-(w+1) (w-1)^2\right),\\
f_2 & = 4 \sqrt{5} z^2 (w+2 z+1),\\
f_3 & = 2 i (1-w) (w+2 z+1)(w+5 z-1),\\
f_4 & = -4 (w-1)^2 (w+2 z+1),\\
g & = \sqrt{5} \left(9-9 w+2 \left(w^2-4 w+3\right) z+15 w^2+w^3\right),\\
\delta & = \sqrt{5} \left(17-9 w+2 \left(w^2-4 w+3\right) z+7 w^2+w^3\right).
\end{align*}
\begin{align*}
F^2_N(z,w) & = (f_1,f_2,f_3,0,0,\ldots, 0, g)/\delta,
\end{align*}
where $f_1,f_2,f_3,g$ and $\delta$ are given in \eqref{eq:F24a}-\eqref{eq:F24b}.

 \begin{lemma}
 \label{lem:rationality}
 The maps $F^1_N$ and $F^2_N$ are not equivalent to polynomial maps from $\Sphere{3}$ into $\Sphere{2N-1}$.
 \end{lemma}

 \begin{proof}
 A straight-forward verification of the polynomiality criterion of Faran--Huang--Ji--Zhang \cite{FHJZ10} leads to the conclusion.
 \end{proof}

 \begin{example}
 In \cite[Example 4.1]{FHJZ10} a family $F_a$ of rational maps of degree $3$ is given, where
 \begin{align*}
 F_a(z,w) = \left(z^2, \sqrt{2} z w, \frac{(z-a)w^2}{1-\bar a z},\frac{\sqrt{1-|a|^2} w^3}{1-\bar a z} \right), \qquad |a|<1.
 \end{align*}
 $F_a$ sends $\Sphere{3}$ into $\Sphere{7}$ and is equivalent to a polynomial map only if $a = 0$. For each $a\neq 0$, the map $F_a$ is $3$-nondegenerate everywhere in $\Sphere{3}$. Hence $F_a$ cannot be equivalent to one of the maps of degree $3$ obtained in  \Cref{thm:deg3families}.
 \end{example}

It can be verified directly, that the families of maps given in \Cref{thm:deg3families} are $(N-4)$-degenerate on an open, dense subset of $\Heisenberg{3}$ and are $(N-3)$-degenerate at $0$. Hence, by \Cref{thm:deg3families}, the only constantly $(2,N-3)$-degenerate maps from $\Heisenberg{3}$ into $\Heisenberg{2N-1}$ for $N\geq 4$ are quadratic maps, which were classified by Lebl \cite{Lebl11}, which shows \Cref{thm:constDegMaps}.\\

For any map $H$ listed in \Cref{thm:deg3families} we define  $X_H=\{q \in \Heisenberg{3}: \tdeg_H (q) = N-3\}$, where $\tdeg_H$ is defined in \Cref{def:tangDeg}. By the homogeneity of the Heisenberg hypersurface and since the maps of \Cref{thm:deg3families} are $(N-4)$-degenerate in a neighborhood of $0$, except at $0$, where they are $(N-3)$-degenerate, our classification shows the following fact:

\begin{lemma}\label{lem:discrete}
For any map $H$ listed in \Cref{thm:deg3families}, the set $X_H$ is nonempty and at most discrete.
\end{lemma}

This lemma immediately implies \Cref{cor:localrigid}. Note that, using translations, as a consequence of \Cref{lem:discrete}, it is possible to construct families of rational $3$-nondegenerate maps from $\Heisenberg{3}$ into $\Heisenberg{7}$.

\bibliographystyle{abbrv}
\bibliography{ref} 
\end{document}

%% file: theproof.tex
Let $H = (f_1,f_2, \ldots, f_{N-1}, g) \colon \mathbb{H}^3 \to \mathbb{H}^{2N-1}$ be one of the maps as in \cref{lem:initalPlus}. Either, we have \textbf{Case I}, when we factor certain components by a quadratic factor,
\begin{equation}\label{ecase1}
	H = \left(f,(z^2 + a w^2) \phi,(a_3 zw + b_3 w^2) \phi, (a_4 zw + b_4 w^2)\phi,0,\ldots,0,g\right),
\end{equation}
where $\phi(0) = 1$ and $\phi^{(0,1)} \in i \,\mathbb{R}$, or, \textbf{Case II}, when $N=4$ and we factor certain components by a linear factor,
\[
	f_2 = (z - \nu w) \xi, \ f_3 = \sigma w \xi,
\]
where $\sigma > 0$, $\nu = \xi^{(0,1)} \in \mathbb{C}$, and $\xi^{(1,0)} = 1$. 
\subsection{Case I: Quadratic factor}\label{sec:caseI}
Here $ \phi(0) = 1$, $\phi^{(0,1)} \in i\, \mathbb{R}$, 
where $a, a_k$, and $b_k$ are constants. Then $H$ sends $\mathbb{H}^3$ into $\mathbb{H}^{2N-1}$ if and only if
\begin{align}\label{me2}
	\frac{g - \bar{g}}{2i} - |f|^2 -  |\phi|^2 \sum_{k=2}^{4} |q_k|^2 
	=
	Q(z,w,\bar{z},\bar{w})\left(\frac{w - \bar{w}}{2i} - |z|^2\right),
\end{align}
with a real-valued function $Q$ which is real-analytic near the origin. 
Setting
\[
	r = |a_3|^2+|a_4|^2 , \quad s = |a|^2 + |b_3|^2 + |b_4|^2, \quad b = a_3 \bar{b}_3 + a_4 \bar b_4,	
\]
we rewrite \eqref{me2} as an equation for $\tilde{H}:= (f,\phi,g)$,
\begin{align}\label{me0}
	\frac{g - \bar{g}}{2i} - |f|^2 -  \left\{|z|^4 + r |zw|^2 + s|w|^4 + 2 \re (a w^2\bar{z}^2 + b zw\bar{w}^2) \right\}|\phi|^2 \notag \\
	=
	Q(z,w,\bar{z},\bar{w})\left(\frac{w - \bar{w}}{2i} - |z|^2\right).
\end{align}
Setting $\bar{z} = \bar{w} = 0$ in \eqref{me0} we find that
\[
	g(z,w) = w Q(z,w,0,0).
\]
In particular, $g(z,0)=0$. To simplify our notations,  we put
\[
	\eta(z,w) = Q_{\bar{w}}(z,w,0,0),\quad \psi (z,w) = \frac{i Q_{\bar{z}}(z,w,0,0)}{2}.
\] 
Differentiating \eqref{me0} with respect to $\bar{z}$ and setting $\bar{z} = \bar{w} = 0$,
\[
	f(z,w) = z Q(z,w,0,0)+w \psi (z,w),
\]
Differentiating \eqref{me0} with respect to $\bar{w}$ and setting $\bar{z} = \bar{w} = 0$, we find that 
\[
	Q(z,w,0,0) = 1+  w \eta (z,w).
\]
Therefore,
\begin{align} \label{e:gsol}
g(z,w) & = w (1+  w \eta (z,w)),\\
\label{e:fsol}
f(z,w) & = z+ w z \eta(z,w)+w \psi (z,w).
\end{align}
In particular $f(z,0)=z$.
Differentiating \eqref{me0} with respect to $\bar{z}$ twice and setting $w=\bar{z} = \bar{w} = 0$, we find that 
\[
	z \phi(z,0) + 2 i\psi (z,0) = 0.
\]
Comparing terms of weights 3 and 4 of the mapping equation \eqref{me0}, we have 
\[
	f^{(1,1)} = \frac{i}{2}, \quad g^{(1,1)} = g^{(0,2)} = 0.
\]
Differentiating \eqref{me0} with respect to $\bar{z}$ and $\bar w$ and setting $w=\bar{z} = \bar{w} = 0$, we find that 
\[
	2z \eta(z,0)-\psi (z,0)+i z = 0,
\]
and with respect to $\bar{w}$ twice,
\[
	-i \eta (z,0)-2 \bar{a} z^2 \phi (z,0)-\lambda  z = 0,
\]
where
\[
	\lambda := f^{(0,2)} \geq 0,
\]
and we obtain from above a system of three linear equations for $(\eta,\psi,\phi)$ along the first Segre set $\Sigma:= \{ w=0\}$, which can be solved explicitly. Precisely,
\begin{equation}\label{e:rsp1}
	\eta(z,0) = \frac{i z (\lambda+2\bar{a} z )}{1 -4 \bar{a} z^2}, 
	\quad
	\psi (z,0) = \frac{i z (1+2\lambda z )}{2(1 -4 \bar{a} z^2)},
	\quad
	\phi (z,0) = \frac{1+2\lambda z}{1 -4 \bar{a} z^2}.
\end{equation}
The latter implies that $\phi^{(1,0)}=2 \lambda$.
Thus, $\widetilde{H}\bigl|_\Sigma = (z,\phi(z,0),0)$ from which we obtain
\begin{lemma}
[First holomorphic functional equation] \label{lem1}
The components $f, \phi$, and $g$ satisfy:
\begin{equation}\label{E1}
	 w f(z,w)-z g(z,w)- \frac12 w^2 (\lambda  w+i z) \phi(z,w) = 0.
\end{equation}
In particular, $f(0,w) = \lambda w^2 \phi(0,w)/2$.
\end{lemma}
\begin{proof} Denote by $E_1(z,w)$ the left hand side of \eqref{E1}, we need to show that $E_1$ is identically zero. For this, we set $\bar{w} = w - 2i z \bar{z}$ into the mapping equation \eqref{me0} to obtain an identity of the form $\Psi(z,\bar{z},w) = 0$, which, when further restricted to $w=0$ gives via direct calculations
\[
	\overline{E_1(z, 2i z \bar{z})} = 0.
\]
This implies $E_1(z,w) \equiv 0$ as the map (of two complex variables) $(z,\bar{z}) \mapsto (z,2i z\bar{z})$ has generic full rank.
\end{proof}

In the next step, we determine $\widetilde{H}_w\bigl|_\Sigma$ which shall be used to establish the second holomorphic functional equation. First, from \eqref{e:fsol} and \eqref{e:rsp1}, we have
\[
	f_w(z,0) = z \eta(z,0) + \psi (z,0) = \frac{i}{2} \frac{z(1 + 4\lambda z + 4 \bar{a} z^2)}{1-4 \bar{a}z^2},
\]
which implies, in particular, that $f^{(2,1)}= 4 i \lambda$.
Differentiating the functional equation \eqref{E1} several times and evaluating at the origin, we can easily find that
\begin{align}
	f^{(0,3)} & = 3i \lambda \mu,\\
	g^{(0,3)} & = 3 \left(f^{(1,2)}+ \mu -2 \lambda ^2\right),
\end{align}
where 
\[
	\mu := -i \phi^{(0,1)} \in \mathbb{R}.
\]

From \eqref{e:gsol} we obtain
\begin{equation}\label{eq:gwwFirstSegre}
g_{ww}(z,0) = 2\eta(z,0) = \frac{2 i z (\lambda + 2\bar{a} z )}{1- 4 \bar{a} z^2}
\end{equation}
 and
\[
	g^{(1,2)} = 2i\lambda.
\]
Also, $g^{(2,1)} =0$ as $g_w(z,0) = 1$, and $g^{(3,0)} = 0$. The coefficient $f^{(1,2)}$ is yet to be determined (equation \eqref{f12} below).

We proceed to determine $\phi_w(z,0)$. Applying $\partial^3 _{\bar{z}}$, $\partial^2_{\bar{z}} \partial_{\bar{w}}$, $ \partial_{\bar{z}} \partial^2_{\bar{w}}$, and $\partial^3_{\bar{w}}$ to the mapping equation \eqref{me0} and evaluating at $w = \bar{z} = \bar{w} = 0$, we find 4 equations (linear over $\mathbb{C}\{z\}$) in the variables
\[
	p(z,w) = Q_{\bar{z}\bar{z}}(z,w), \quad t(z,w) = Q_{\bar{z}\bar{w}}(z,w), \quad q(z,w) = Q_{\bar{w}\bar{w}}(z,w,0,0)
\]
along the first Segre set. Precisely,
\begin{equation}
	\begin{pmatrix}
		1 & 0 & 0  \\
		i & -4 z & 0 \\
		0 & i & -z  \\
		0 & 0 & 1  \\
	\end{pmatrix}
	\begin{pmatrix}
	p \\ t\\ q 
\end{pmatrix}
	=
	\begin{pmatrix}
		4 \lambda  z \phi  \\
		4 i \mu  z^2 \phi +8 i \lambda  z \\
		-4 \bar{a} \lambda  z^2 \phi -z \bar{f}^{(1,2)}+ \lambda  \\
		-4 \bar{a} \mu  z^2 \phi - 2 \lambda \mu z - \bar{f}^{(1,2)}+2 \lambda^2\\
	\end{pmatrix}
\end{equation}
holds along $\Sigma := \{w =0\}$. This overdetermined system is uniquely solvable. In other word, the over-determinacy does not give any new constraints on the 3-jets of the map. From the first row, we have
\[
	p(z,0) = 4\lambda z \phi (z,0).
\]
From the remaining rows, we can easily find
\[
	t(z,0)  = \frac{-i \left(\lambda + z\left(\mu -2\lambda ^2\right) + 2\lambda  z^2 (\mu -4\bar{a})\right)}{1-4 \bar{a} z^2}, 
	\quad
	q(z,0)= \bar{f}^{(1,2)} -2 \lambda ^2 +\frac{\mu(1+ 2\lambda  z)}{1- 4 \bar{a} z^2}.
\]

By applying the differential operator $\partial_w \partial_{\bar{w}}^j \partial_{\bar{z}}^{2-j}$ to the mapping equation and substituting $p(z,0) = 4\lambda z \phi(z,0)$ we obtain along $\Sigma$,
\begin{equation}\label{eq:etapsiwFirstSegre}
	\begin{pmatrix}
	0 & 1 & -z \\
	2z & -i & 0 \\
	2i & 0 & 4 \bar{a} z^2 \\
	\end{pmatrix}
	\begin{pmatrix}
	\eta_w\\
	\psi_w\\
	\phi_w
	\end{pmatrix}
	=
	\begin{pmatrix}
		-i \lambda  \phi \\
		2 r z \phi -i t-if_w \\
	    iq- 2\lambda  f_w  -4 b z \phi
	\end{pmatrix}.
\end{equation}
Solving for $\phi_w$ we obtain,
\begin{equation}\label{eq:phiwFirstSegre}
\phi_w = \frac{\left(4 b z^2 +2 i r z+i \lambda\right) \phi +(1+2 \lambda  z) f_w-i z q+t}{z(1-4 \bar{a} z^2)}.
\end{equation}
The singularity $z=0$ of the right-hand side is removable, since $t(0,0) = -i \lambda$. 
Using the above formulas for $\phi, f_w, q$ and $t$ along $\Sigma$ we find that
\[
	i \mu = \phi_w(0,0)
	=
	\frac{1}{2} i \left(1-2 \bar{f}^{(1,2)}+12 \lambda ^2-4 \mu +4 r\right),
\]
which amounts to
\begin{equation}\label{f12}
f^{(1,2)} =  6 \lambda ^2-3 \mu +2 r+\frac{1}{2}.
\end{equation}
In particular, $f^{(1,2)}$ is real. From \eqref{e:fsol}, we can easily determine $f_{ww}(z,0)$:
\begin{align}\label{fwwFirstSegre}
	f_{ww}(z,0) = 2(z \eta_w(z,0) + \psi_w(z,0)).
\end{align}
At this point, it is possible to establish another holomorphic functional equation. But the formula is very complicated.

In the next step, we determine $\phi_{ww}(z,0)$, which, together with \eqref{eq:gwwFirstSegre} and \eqref{fwwFirstSegre}, will be used to establish the third holomorphic functional equation.  Observe that, from the formulas for $f(z,0)$, $f_w(z,0)$ and $f_{ww}(z,0)$, we find that
\[
	f^{(4,0)} = 0, \quad f^{(3,1)} = 24 i \bar{a}, \quad f^{(2,2)} = 16 i b -4 \lambda  (2 \bar{a}-3 \mu +2 r+2),
\]
and by differentiating \eqref{E1} five times with respect to $w$ and evaluating at the origin,
\[
	f^{(0,4)}= 6 \lambda  \phi^{(0,2)}.
\]
We also determine the fourth order derivatives of $g$ at the origin, namely
\[
	g^{(4,0)} = g^{(3,1)} = 0, \quad g^{(2,2)} = 8i \bar{a}, \quad g^{(1,3)}= 3 (4 i b+2 \lambda  \mu -\lambda ),
\]
where the last equality follows from $g_{www}(z,0) = 6 \eta_w(z,0)$, which is obtained from \eqref{e:gsol}.
By differentiating \eqref{E1} four times with respect to $w$ and once with respect to $z$, we find that 
\[
	g^{(0,4)} = 4 \left(f^{(1,3)}-3 \lambda  \phi^{(1,1)}\right)-6 i \phi ^{(0,2)}.
\]
Regarding the second-order derivatives of $\phi$ at the origin, by considering \eqref{e:rsp1} and \eqref{eq:phiwFirstSegre}, we obtain
\begin{align}
	\phi^{(2,0)} = 8\bar{a}, \quad \phi^{(1,1)} = 4 b+i \lambda  (8 \bar{a}-4 \mu +4 r+3).
\end{align}


The coefficients $f^{(1,3)}$ and $\phi^{(0,2)}$ will be determined in the following. We apply $\partial^j_{\bar{z}} \partial^{4-j}_{\bar{w}}$ to the mapping equation to obtain a system of linear equations in the variables
\begin{align*}
	k(z,w) = Q_{\bar{z}\bar{z}\bar{z}}(z,w,0,0), \ l(z,w) = Q_{\bar{z}\bar{z}\bar{w}}(z,w,0,0), \\ m(z,w)=Q_{\bar{z}\bar{w}\bar{w}}(z,w,0,0), \ n(z,w) = Q_{\bar{w}\bar{w}\bar{w}}(z,w,0,0),
\end{align*}
along $\Sigma$, given by
\begin{equation}
	\begin{pmatrix}
	1 & 0 & 0 & 0\\	
	i & -6 z & 0 & 0 \\
	0 & i & -2 z & 0 \\
	0 & 0 & 3i & -2z \\
	0 & 0 & 0 & 4 \\
	\end{pmatrix}
	\begin{pmatrix}
	k\\l\\
	m\\
	n
	\end{pmatrix}
	=
	\begin{pmatrix}
	24a z \phi \\
	48 i a z - 12 z^2 \bar{\phi}^{(1,1)} \phi  \\
	4 a - \bar f^{(2,2)} z-(2 \bar \phi^{(0,2)}+16 |a|^2 ) z^2\phi  \\
	i \bar{g} ^{(1,3)} -2z \bar{f}^{(1,3)} - 12 \bar{a}   \bar{\phi} ^{(1,1)} z^2 \phi   \\
	 \bar{g} ^{(0,4)} + 12 i \lambda \bar \phi^{(0,2)} z -24\bar{a} \bar{\phi} ^{(0,2)} z^2 \phi 
	\end{pmatrix}.
\end{equation} 

This overdetermined system can be solved uniquely using the relations for $\phi(z,0)$, $\phi^{(1,1)}$, $f^{(2,2)}$, $g^{(1,3)}$ and $g^{(0,4)}$ established above. In other words, the solvability of this system does not give any new constrains. 

From the first row, we can easily find
\[
	k(z,0) = 24az \phi(z,0),
\]
while from the last row, we can easily solve for $l(z,0)$ and $n(z,0)$, given by
\begin{align*}
l(z,0) & = - 8 i a  + 2( i a + \bar \phi^{(1,1)} z) \phi(z,0),\\
n(z,0) & = \frac{\bar g^{(0,4)}}{4} + 3 i \lambda \bar \phi^{(0,2)} z + 6 i \bar a \bar \phi^{(0,2)} z^2 \phi(z,0).
\end{align*} 
The formula for $m$ is more complicated:
\begin{align*}
m(z,0) = \frac{1}{6} \left(2 \bar g^{(1,3)} -i \left(\bar g^{(0,4)}-4 \bar f^{(1,3)}\right)z + 12 \lambda  z^2 \bar \phi^{(0,2)}\right) + 4 \bar a  \left(i \bar \phi^{(1,1)} + z \bar \phi^{(0,2)}\right) z^2 \phi(z,0).
\end{align*}


Similarly, applying $\partial_w \partial_{\bar{z}}^{3-j} \partial_{\bar{w}}^{j}$ to the mapping equation, we obtain the following overdetermined system
\begin{align}\label{eq:ptqwFirstSegre}
	\begin{pmatrix}
	6 z & 0 & 0 \\
	i & -4 z & 0 \\
	0 & 2 i & -2 z \\
	0 & 0 & 3 
	\end{pmatrix}
	\begin{pmatrix}
	p_w\\
	t_w\\
	q_w
	\end{pmatrix} 
	 =
	\begin{pmatrix}
	  24 \lambda   z^2 \phi_w- ik  \\
	 i l+8 i \lambda  f_w -8 \lambda  r z \phi + 4 i \mu z^2 \phi_w  \\
	 i m - 2 \bar{f}^{(1,2)} f_w -8\bar{a} \lambda  z^2 \phi_w+ 4z  (ir\mu -2b \lambda) \phi\\
	 n + 6 \lambda \mu  f_w + 12 \bar{a} \mu z^2 \phi_w + 12 b\mu z \phi  
	\end{pmatrix}
\end{align}

To establish the formula for $\phi_{ww}$ along $\Sigma$ we apply $\partial^2_w \partial_{\bar z}^j \partial_{\bar w}^{2-j}$ to the mapping equation to obtain the following system
\begin{align}\label{eq:psietaphiwwFirstSegre}
\begin{pmatrix}
	4 i z & 0 & 2 z^2 \\
	-2 & 2 z & 0 \\
	0 & 2 i & -2 z \\
	0 & i & 2 \bar a z^2 
	\end{pmatrix}
	\begin{pmatrix}
	\psi_{ww}\\
	\eta_{ww}\\	
	\phi_{ww}
	\end{pmatrix} 
	 =
	\begin{pmatrix}
	 4 a \phi - i p_w  \\
	 -4 \bar b \phi  - 4 r z \phi_w + 2 i t_w + i f_{ww}\\
	 4 s \phi + 4 b z \phi_w   - i q_w  + \lambda f_{ww}
	\end{pmatrix}
\end{align}

Solving for $\phi_{ww}$ we have,
\begin{align}\label{eq:phiwwFirstSegre}
\phi_{ww} = \frac{R + 8 s \phi  + 4(2 b z + i r)\phi_w- 2 i q_w  + 2 \lambda f_{ww} }{1-4 \bar a z^2},
\end{align}
where
\begin{align*}
R=\frac{1}{2 z^2} \bigl(i p_w - 4 a \phi + 2(4 i \bar b \phi + 2 t_w  + f_{ww})z\bigr).
\end{align*}
Since $\phi(z,0) = 1 + 2 \lambda z + o(1), p_w(z,0) = - 4 i a - 4 i \lambda (2 a - \mu) z + o(1), t_w(0,0) = \lambda \mu - \lambda/2 - 2 i \bar b$, the expression $R$ has a removable singularity at $z=0$. 

The solvability of the system \eqref{eq:ptqwFirstSegre} gives new constrains on the coefficients of order $4$ and lower. Precisely, by the Kronecker--Capelli theorem, we have the vanishing of the determinant $\Delta$ of the augmented matrix of \eqref{eq:ptqwFirstSegre}, which can be computed explicitly:
\[
	\Delta
	=
	\frac{-2z^3(\Delta_0 + z \Delta_1 + z^2 \Delta_2)}{1-4\bar{a}z^2}
\]
which, in turns, gives an overdetermined system for $\bar{f}^{(1,3)}$ and $\bar{\phi}^{(0,2)}$, namely
\begin{align}
\label{eq:Phi02F13}
	&\begin{pmatrix}
	12 i & 4 \\
	2 i \lambda  & 0 \\
	0 & -4 \bar{a} \\
	\end{pmatrix}
	\begin{pmatrix}
		\bar{\phi}^{(0,2)} \\
		\bar{f}^{(1,3)}
	\end{pmatrix} \notag \\
	&=
	\begin{pmatrix}
	3 i \left(32 |a|^2-4 (24a+8 \bar{a} -16\mu +12r +5) \lambda ^2-48 i \bar{b}  \lambda +4 \mu ^2-6 \mu + 4r \left(1-2 \mu\right)+1\right) \\
	-i \left(\lambda  \left(-8 \lambda ^2+4 \mu  (\mu +1)-4 r-1\right)+8 \bar{a}  \left(\lambda  \left(2 a+2 \lambda ^2-4 \mu +2 r+1\right)+2 i \bar{b} \right)\right) \\
	3 i \bar{a}  \left(4 \lambda  ((8 a+5) \lambda +4 i \bar{b} )-2 \mu  (6 \mu +1)+4 r \left(4 \lambda ^2+2 \mu +1\right)+1\right) \end{pmatrix}.
\end{align}
Applying the Kronecker--Capelli theorem to the overdetermined system \eqref{eq:Phi02F13} gives
\begin{lemma} If there exists a map $H$ satisfying the mapping equation \eqref{me0} and normalizing conditions, then
\begin{equation}\label{e:1}
	a \left(a-\lambda ^2\right) (2 b+i \lambda  (4 \bar{a}-4 \mu +2 r+1)) = 0.
\end{equation}
\end{lemma}
\begin{proof} The left hand side of \eqref{e:1} is a multiple of the determinant of the augmented matrix of \eqref{eq:Phi02F13} and hence the lemma follows from the Kronecker--Capelli theorem.
\end{proof}

If \eqref{e:1} is satisfied and $a\ne 0$, then $\bar{f}^{(1,3)}$ and $\bar{\phi}^{(0,2)}$ can be solved uniquely from lower order coefficients. If $a=0$, then $\bar{f}^{(1,3)}$ can be solved from $\bar{\phi}^{(0,2)}$. From \eqref{e:1}, we divide our following steps into several cases.


\subsubsection{Case 1. $a = 0$} In this case, we have
\[
	\eta (z,0) = i \lambda  z,\quad \psi (z,0) = \frac{1}{2} i z (2 \lambda  z+1),\quad \phi(z,0) = 2 \lambda  z+1
\]
and 
\begin{align}
	f_w(z,0) &=  2i\lambda z^2 + \frac{i}{2}z,  \\
	\phi_w(z,0) &= 4 \lambda  z^2 (2 b+i \lambda )+z (4 b-4 i \lambda  \mu +3 i \lambda +4 i \lambda  r)+i \mu, \\
	g_w(z,0) &= 1.
\end{align}
From these formulas, we can obtain 
\begin{lemma}[Second holomorphic functional equation]\label{lem5} If $a= 0$, then the components $f$ and $\phi$ satisfy
\begin{align}\label{eq:E2}
		A_1 (z,w) f(z,w) + B_1(z,w) \phi(z,w) = C_1(z,w),
\end{align}
where
\begin{align}
	A_1(z,w) &= 2 \left(2 \lambda  w^2+i w z+2 z\right), \\
	B_1(z,w) & = w \bigl(2 w z^2 (\mu-2 r)+i \lambda  (4 \mu -3) w^2 z-2 \lambda ^2 w^3-4 i z^2-6 \lambda  w z\bigr),\\
	C_1(z,w) & = 4z^2.
\end{align}
\end{lemma}
\begin{proof}
Write $E_2=A_1 f + B_1 \phi - C_1$. We need to show that $E_2(z,w)=0$. First we set $w = \bar w + 2 i z \bar z$ in the mapping equation \eqref{me0} to obtain an equation of the form $\Psi(z,\bar z, \bar w)=0$. Applying the CR vector field $L$ and further restricting to $\bar w =0$ leads to an equation of the form $\overline{E_2(z,2i z \bar z)}=0$ after using previously obtained formulas and some direct computations. Since $(z,\bar z) \mapsto (z, 2 i z \bar z)$ is generically of full rank the claim follows.
\end{proof}

From this point on, we divide Case 1 into two subcases regarding whether $\lambda \ne 0$ or $\lambda=0$:
\subsubsection*{Subcase 1.1: $\lambda \ne 0$} 
\begin{lemma}\label{lem6} Suppose that $a = 0$ and $\lambda \ne 0$. If $H$ satisfies the mapping equation, then either $H$ is quadratic or 
	\begin{align}
		b & = - \frac{i \lambda }{4}, \\
		r & = \frac{1}{4} (6\mu-1), \\
		s & = \frac{1}{4} \left(2 \mu ^2-\lambda ^2\right).
	\end{align}
\end{lemma}
Before we give the proof of \cref{lem6}, we deduce all the maps in this case.
\begin{corollary}\label{cor:mapsCase11}
Let $N \geq 4$ and suppose $a = 0$ and $\lambda \neq 0$. Then $H$ is either quadratic or equivalent to a map as given in \Cref{thm:deg3families} (i).
\end{corollary}

\begin{proof}[Proof of \Cref{cor:mapsCase11}]
It holds that for any set of complex parameters $\tilde a = (a_3,a_4)$ and $\tilde b = (b_3,b_4)$ satisfying
\begin{align}\label{eq:paramCase11}
	\sum_{k=3}^{N-1}|a_k|^2 = \frac{1}{4} (6\mu-1), \quad \sum_{k=3}^{N-1}|b_k|^2 = \frac{1}{4} \left(2 \mu ^2-\lambda ^2\right), \quad \sum_{k=3}^{N-1} a_k \overline{b_k}  = - \frac{i \lambda }{4},
\end{align}
the map
\[
	H = (f, z^2 \phi, w(a_3 z + b_3 w) \phi,w(a_3 z + b_3 w) \phi,w(a_4 z + b_4 w) \phi, 0, \ldots, 0, g)/\delta,
\]
where $f,\phi, g$ and $\delta$ are as in the corollary, sends $\mathbb{H}^3$ into $\mathbb{H}^{2N-1}$.
From \cref{lem6} we have $\mu \geq  1/6$ and by the definitions of $r,s$ and $b$, it follows from the Bunyakovsky--Cauchy--Schwarz inequality that, 
\begin{align*}
\frac{\lambda^2}{16}=|b|^2 \leq r s = \frac{1}{16}(6 \mu-1)(2 \mu^2-\lambda^2),
\end{align*}
hence
\[
0< \lambda \leq \sqrt{\frac{\mu(6\mu-1)}{3}}.
\]
In particular $\mu > 1/6$.
When $N=4$, the second inequality becomes an equality. 
The conditions from \eqref{eq:paramCase11} give the desired formulas for the remaining coefficients.
\end{proof}

We proceed to prove \cref{lem6}. We have $g_{ww}(z,0) = 2i\lambda z$ from \eqref{eq:gwwFirstSegre} and further compute $\eta_w$, $\psi_w$ along $\Sigma$ from \eqref{eq:etapsiwFirstSegre} and use \eqref{fwwFirstSegre} to deduce that
\[
	f_{ww}(z,0)
	=8 i \lambda  z^3 (2 b+i \lambda )+2 z^2 (4 i b+3 \lambda  \mu -2 \lambda -2 \lambda  r)+\frac{1}{2} z \left(12 \lambda ^2-6 \mu +4 r+1\right)+\lambda.
\]

Using this formula and the explicit expressions of \eqref{eq:ptqwFirstSegre} we obtain from \eqref{eq:phiwwFirstSegre}
\begin{align}\label{eq:phiwwSubcase11}
\nonumber
\phi_{ww}(z,0) & = 16 \lambda   (i \lambda + 2 b)^2 z^3 + 4 \left(8 b^2 - 2 i b
   \lambda  \bigl(8 (\mu -r) - 7\bigr) + \lambda ^2 (7 \mu - 6 r - 4)\right) z^2 \\
   & \quad +  \Bigl(\lambda  (10 \bar \phi^{(0,2)}-36 \lambda ^2+8 \mu^2+24 \mu -16 r^2 + 4 r (6 \mu -7) + 16 s-7) \\
\nonumber
   & \quad  + 8 i b (4 \lambda ^2-\mu +2 r + 1)\Bigr)z -2 i \bar f^{(1,3)} + 5 \bar \phi^{(0,2)} +8 i \lambda (2 b   + 7  \bar b ) + 24 \lambda^2 (1- 2\mu) \\
\nonumber
   & \quad + \mu(3 - 2 \mu) + 2 r (20 \lambda^2 + 2 \mu- 1) + 8 s - 1/2.
\end{align}

With these formulas available we are able to obtain the following equation.

\begin{lemma}[Third holomorphic functional equation]\label{lem7}
	If $a= 0$, and $\lambda \ne 0$, then the components $f$ and $\phi$ satisfy
	\begin{equation}\label{eq:E3}
		A_2 (z,w) f(z,w) + B_2(z,w) \phi(z,w) = C_2(z,w),
	\end{equation}
	where
	\begin{align*}
	A_2(z,w) &= -8 \lambda  z^3+2 i w z^2 \left(12 \lambda ^2-6 \mu +4 r+1\right)+4 w^2 z (3 \lambda  \mu -4 i \bar{b} -2 \lambda -2 \lambda  r)\\
		& \qquad -8 \lambda  w^3 (2 \bar{b}-i \lambda )-4 z^2+16 i \lambda  w z, \\
		B_2(z,w) & =-16 b w z^4-w^2 z^3 (-48 i b \lambda +32 i \bar{b}   \lambda -96 \lambda ^2 \mu +20 \lambda ^2-12 \mu ^2+8 \mu +64 \lambda ^2 r\\
		 & +8 \mu  r-4 r-1)  -2 w^3 z^2 (16 \bar{b}   \lambda ^2-8 \bar{b}   \mu +4 \bar{b}  +2 i \lambda ^3-6 i \lambda  \mu ^2+2 i \lambda  \mu -i \lambda \\
		 & +4 i \lambda  \mu  r+2 i \lambda  r) 2 \lambda  w^4 z (-8 i \bar{b}   \mu +8 i \bar{b}  -7 \lambda  \mu +4 \lambda +2 \lambda  r)+4 \lambda ^2 w^5 (2 \bar{b}  -i \lambda ) \\
		 & -8 i w z^3 (r-\mu ) 2 w^2 z^2 (8 i \bar{b}  -12 \lambda  \mu +9 \lambda +4 \lambda  r)+16 \lambda  w^3 z (\bar{b}  -i \lambda )+4 z^3\\
		 & -12 i \lambda  w z^2,\\
		C_2(z,w) & = 4 i \lambda  w z^2.
	\end{align*}
\end{lemma}
\begin{proof}
We proceed similarly as in the proof of the second holomorphic functional equation \cref{lem5}: We write $E_3 = A_2 f + B_2 \phi - C_2$ and need to show that $E_3(z,w)=0$. First we set $w = \bar w + 2 i z \bar z$ in the mapping equation \eqref{me0}, then apply the CR vector field $L$ twice and restrict to $\bar w = 0$ to obtain, after some straight-forward computations, an equation of the form $\overline{E_3(z,2i z \bar z)}=0$, which shows the claim.
\end{proof}

Proceeding to prove \cref{lem6}, using \cref{lem5} and \cref{lem7}, we write
\begin{align*}
M \coloneqq \left(\begin{array}{cc}
A_1 & B_1\\
A_2 & B_2
\end{array} \right),
\end{align*}
where $A_1,B_1, A_2$ and $B_2$ are as in these two lemmas, such that
\begin{align}\label{eq:systemCase11}
M
	\begin{pmatrix}
		f\\
		\phi
	\end{pmatrix}
	= 
	-4z^2
	\begin{pmatrix}
		1 \\
		-i\lambda w
	\end{pmatrix}.
\end{align}
This system determines $f$ and $\phi$ uniquely, but, for general parameters, its solution may not be holomorphic at the origin. Moreover, even for certain values of parameters, the solutions are holomorphic but they do not comes from a genuine sphere map. We will determine the relevant parameters such that this system gives a unique holomorphic solution $(f,\phi)$.

By assumption $\lambda \neq 0$, hence $f(0,w) \ne 0$. Thus, from \eqref{eq:systemCase11} it follows, that the determinant $\Delta$ of the coefficient matrix $M$ must be divisible by $z$. In order to compute the term of lowest degree in $\Delta$ we decompose the entries of $M$ into homogeneous parts
\[
	A_1 = A_1^{(1)} + A_1^{(2)}, \ B_1 = B_1^{(3)} + B_1^{(4)}, \ A_2 = A_2^{(2)} + A_2^{(3)}, \  B_2 = B_2^{(3)} + B_2^{(4)} + B_2^{(5)}.
\]
Then we have,
\[
	\Delta = A_1^{(1)} B_2 ^{(3)} - A_2^{(2)} B_1^{(2)} + h.o.t. = 16 z^3 (z - 3 i\lambda w) + h.o.t.
\]
Thus, $\Delta^{(4)} = 16 z^3 (z - 3 i\lambda w)$. The higher order terms are of total degree at most 7.

By Cramer's rule, we have $\phi = \Delta_{\phi}/\Delta$, where 
\[
	\Delta_{\phi} = \det \begin{pmatrix}
		A_1 & C_1 \\
		A_2 & C_2
	\end{pmatrix}
		=
		-4 z^2 (A_1(-i \lambda w) - A_2)
\]
Since $\phi$ has no singularity at the origin, we can write
\[
	\phi = \frac{P_\phi}{Q_\phi}
\]
with $P_{\phi}, Q_{\phi} \in \mathbb{C}[z,w]$ and $Q_{\phi}(0,0) \ne 0$. Thus,
\[
	P_\phi \Delta = Q_{\phi} \Delta_{\phi}.
\]
If $D$ is any irreducible divisor (in $\mathbb{C}[z,w]$) of $\Delta$ such that $D(0,0) =0$, then $D$ must be a divisor of $\Delta_{\phi}$, since $Q_{\phi}(0,0) \ne 0$. Thus, we can find a common divisor $D_0$ of $\Delta$ and $\Delta_{\phi}$ with $D_0(0,0) = 0$ and of maximal degree. Then $\deg D_0 \leq \deg \Delta_{\phi} = 5$ and, by the maximality, $\deg D_ 0 \geq 4$ (as the lowest degree terms in $\Delta$ are of degree $4$).

If $\deg D_0 = 5$, then we can easily check that $H$ is quadratic and we are done.

Assume that $\deg D_0 = 4$, then, up to a non-zero factor, $D_0$ agrees with the terms of degree four of $\Delta_\phi$. Thus, we can assume that 
\[
	D_0 = z^3(z-3i \lambda w),
\]
and the terms of degree five $\Delta_{\phi}^{(5)}$ of $\Delta_{\phi}$ must be divisible by $D_0$. Explicitly,
\[
	\Delta_{\phi}^{(5)}
	=
	8z^2\bigl(2\lambda   (i\lambda -4 \bar{b})w^3-4 \lambda  z^3 + i w z^2 \left(12 \lambda ^2-6 \mu +4 r+1\right)+ w^2 z (6 \lambda  \mu-8 i \bar{b} -3 \lambda -4 \lambda  r)\bigr)
\]
It follows that the coefficient of $z^2 w^3$ in $\Delta^{(5)}_\phi$, must vanish and thus, as $\lambda \ne 0$,
\[
	b = -\frac{i}{4} \lambda,
\]
as desired.

The vanishing of $\Delta_{\phi}^{(5)}$ along $\{z-3i \lambda w = 0\}$ yields
\[
	1-6 \mu +4 r = 0,
\]
which gives,
\[
	r = \frac{1}{4} (6 \mu -1).
\]
Plugging these into the mapping equation, we easily find the desired formula for $s$. We can then check directly that the resulting map is the cubic map as given above and sends $\mathbb{H}^3$ into $\mathbb{H}^{2N-1}$. This finishes subcase 1.1.

\subsubsection*{Subcase 1.2: $a=\lambda = 0$} Proceeding as in subcase 1.1, it must hold that
\begin{align*}
	b & = 0,\\
	r & = \mu(\mu^2 + \gamma), \\
	s & = \gamma + \mu(\mu+1) - \frac{1}{4}.
\end{align*}
where $\gamma = \phi^{(0,2)}/2 \in \R$, which results in a two-parameter family of quadratic maps and we are done in subcase 1.2 and hence in case 1.

\subsubsection{Case 2: $a\ne 0$}

 From \eqref{e:1} above, we have two subcases.
\subsubsection*{Subcase 2.1: $a = \lambda^2 \ne 0$ (in particular, $a$ is real)}

In this case, $\widetilde{H}\bigl| _{\Sigma}$ reduces to
\[
	\widetilde{H}\bigl| _{\Sigma} = \left(z,\ \frac{1}{1-2\lambda z},\ 0\right).
\]

When $a\ne 0$, we can solve in \eqref{eq:Phi02F13} to obtain
\begin{align}
	f^{(1,3)} & = \frac{3i}{4} \left(1-16 i b \lambda +32 \lambda ^4+20 \lambda ^2-12 \mu ^2-2 \mu +4 r \left(4 \lambda ^2+2 \mu +1\right)\right), \\
	\phi^{(0,2)} & = 8 i b \lambda +16 \lambda ^2 \mu -16 \lambda ^4-2 \mu ^2-2 \mu +\left(2-8 \lambda ^2\right) r+\frac{1}{2}.
\end{align}
\begin{lemma}[Second holomorphic functional equation] If $a = \lambda^2 \ne 0$, then the components $f$ and $\phi$ satisfy:
	\begin{equation}
		A_1(z,w) f(z,w) + B_1 (z,w) \phi(z,w) = C_1(z,w),
	\end{equation}
	where
	\begin{align}
		A_1(z,w) & = 4(z + i\lambda w) + 2iw(z - i\lambda w), \\
		B_1(z,w) & = 2w (\lambda  w-2i z) (z-i\lambda z) \notag  \\
		&\quad   - w^2(2 z^2 (2 r-\mu )+i \lambda  w z \left(8 \lambda ^2-6 \mu +4 r+3\right)+\lambda ^2 w^2), \\
		C_1(z,w) & = 4 z (z+i \lambda  w).
	\end{align}
\end{lemma}
\begin{proof}
The proof is analogous to  \cref{lem5} in case 1. 
\end{proof}
\begin{lemma}[Third holomorphic functional equation] If $a = \lambda^2 \ne 0$, then the components $f$ and $\phi$ satisfy:
	\begin{equation}
		A_2(z,w) f(z,w) + B_2 (z,w) \phi(z,w) = C_2(z,w),
	\end{equation}
	where
	\begin{align}
		A_2(z,w) & = -4 (z-i \lambda  w) \left(\lambda ^2 w^2-2 i \lambda  w z+z^2\right) + O(4) \\
		B_2(z,w) & = 4 z^2 (z-i \lambda  w)^2 + O(5)\\
		C_2(z,w) & = 4 i \lambda  zw \left(z^2+\lambda ^2 w^2\right).
	\end{align}
\end{lemma}
\begin{proof}
The proof is analogous to  \cref{lem7} in subcase 1.1. Note that under the conditions in this case, there is a factor $4 (z-i \lambda  w)$ that has been canceled out.
\end{proof}
Combining the second and third holomorphic functional equation gives a system of linear equations that determines $f(z,w)$ and $\phi(z,w)$ uniquely, namely
\[
	\begin{pmatrix}
		A_1 & B_1 \\
		A_2 & B_2
	\end{pmatrix}
	\begin{pmatrix}
		f\\
		\phi
	\end{pmatrix}
	=
	4z(z+i \lambda w)
	\begin{pmatrix}
		1\\
		i\lambda w(z - i\lambda w)
	\end{pmatrix}.
\]
Let $\Delta(z,w)$ be the determinant of the coefficient matrix on the left-hand side.
Observe that $f(-i\lambda w, w) = -i\lambda w + o(w^2) \ne 0$ as $\lambda \ne 0$, hence it must hold that
\[
	\Delta(-i\lambda w, w) = 0.
\]
Collecting the terms of total degree 7 in $\Delta$, we have
\[
	\Delta^{(7)}
	=
	A_1^{(2)} B_2^{(5)} + A_1^{(1)} B_2^{(6)} - A_2^{(3)} B_1^{(4)} - A_2^{(4)} B_1^{(3)},
\]
and by direct calculation, we have 

\[
	\Delta^{(7)}(-i\lambda w, w) = -32w^7 \lambda^4 \left(2 \bar{b}-i \lambda  \left(4 \lambda ^2-4 \mu +2 r+1\right)\right).
\]
Therefore, we obtain
\begin{lemma}\label{lem10}
	If $a = \lambda^2 \ne 0$, then
	\[
	b = -\frac{i\lambda }{2}  \left(4 \lambda ^2-4 \mu +2 r+1\right).
	\]
\end{lemma}
Since $a = \lambda^2 \in \mathbb{R}$, this case is a special case of Subcase 2.2 below.
\subsubsection*{Subcase 2.2:} Here we assume
\[
	b =  -\frac{i \lambda }{2} (4 \bar{a}-4 \mu +2 r+1).
\]
In this case, we have
\begin{align}
	f^{(1,3)} & = \frac{3i}{4}  \left(4 \lambda ^2 (8 \mu +3)-12 \mu ^2-2 \mu +(8 \mu +4) r+1\right),\\
	\phi^{(0,2)} & = 8 |a|^2 -8 a \lambda ^2+4 \lambda ^2-2 \mu ^2-2 \mu +2 r+\frac{1}{2}.
\end{align}
This case leads to the following:
\begin{lemma}\label{lem11} If $a \ne 0$, then either $H$ is quadratic or exactly one of the following two possibilities occurs:
\begin{enumerate}
	\item $\lambda = 0$ and
	\[
		b=0, \quad r = \frac{1}{12} (16\mu -3), \quad s =\frac{1}{9} \mu^2, \quad a = \frac{1}{3} \mu e^{i\theta},
	\]
	with $\mu > 3/16$ and $\theta \in \R$.
	\item $\lambda \ne 0$ and
	\begin{align}
		b =i \lambda r, \quad 
		r = -\frac{1}{4}(1+16a) ,\quad 
		s  = a^2 + \lambda^2 r,\quad \mu = - 3a.	
	\end{align}
	with $a \leq -1/16$ and $\lambda>0$. 
\end{enumerate}
\end{lemma}

This allows to deduce the following maps.

\begin{corollary}
Suppose $a \neq 0$, then $H$ is either quadratic or a map as given in \Cref{thm:deg3families} (ii).
\end{corollary}


\begin{proof}
When $\lambda = 0$, using \cref{lem11} (1), it follows from the formula for $s$, that $b_3 = b_4  =0$. We apply  
a unitary matrix to the components $f_3$ and $f_4$ to achieve that $(a_3,a_4) =(\sqrt{(16\mu -3)/12},0)$, rotations of the form $(z,w)\mapsto (e^{i\theta} z, w)$, for some $\theta \in \R$, 
such that a reparametrization of $a$ and $\mu$ gives the desired map in \Cref{thm:deg3families} (ii) with $\lambda =0$.

When $\lambda \neq 0$ and $N\geq 4$, we observe that using \cref{lem11} (2) 
and the Bunyakovsky--Cauchy--Schwarz inequality we have
\begin{align*}
\lambda^2 r^2 = |b|^2 \leq r (s-|a|^2) = \lambda^2 r^2,
\end{align*}
which means that the complex vectors $\tilde a = (a_3, a_4)$ and $\tilde b = (b_3, b_4)$ are extremal for the Bunyakovsky--Cauchy--Schwarz inequality, hence 
\[
	b_k = -i \lambda a_k, \quad k = 3,4.
\]
Applying a unitary matrix to the components $f_3$ and $f_4$ we obtain that $\tilde a = (\sqrt{-(1+16a)}/2,0)$, which gives the desired map in \Cref{thm:deg3families} (ii) with $\lambda \neq 0$.

Note that we require $a<-1/16$, since otherwise the third component of the map vanishes.
\end{proof}

\begin{proof}[Proof of \cref{lem11}]
In this case, we have the following system from the second and third holomorphic functional equation for $f$ and $\phi$:
\begin{align}\label{eq:systemCase22}
\begin{pmatrix}
	A_1 & B_1 \\
	A_2 & B_2
\end{pmatrix}
\begin{pmatrix}
	f\\
	\phi
\end{pmatrix}
= 
\begin{pmatrix}
	C_1\\
	C_2
\end{pmatrix},
\end{align}
where the coefficients are
\begin{align*}
 A_1 &= 4(z^2 + a w^2) +2 w z (2 \lambda  w+i z)-2 i a w^3,\\
 B_1 &= -2 w (a \lambda  w^3+3 \lambda  w z^2+2 i z^3) \\
	 & + w^2(i a \lambda  w^3-2 z^3 (2 r-\mu )-i \lambda  w z^2 (8 a-4 \mu +3),  \\
 	&-2 w^2 z (4 a \lambda ^2-3 a \mu +2 a r+a+\lambda ^2 )),\\
C_1 & = - 4 z (a w^2+z^2),\\
A_2 & = - 2 i a^2 w^5 + w^4 (2 z (8 a^2 \lambda - 2 a \lambda  (\mu - 2 r -2)) + 4 a^2) \\
 	& + 2 w^3 z^2 (2 i a (6 \lambda ^2-7 \mu +6 r+2) - 8 i \lambda ^2 (r - 2 \mu )) \\
 	&  + w^2 (2 z^3 (4 a \lambda -10 \lambda  \mu +4 \lambda r) + 16 a z^2) \\
 	& + w (2 z^4 (12 i \lambda ^2-6 i \mu +4 i r + i) + 16 i \lambda  z^3) - 8 \lambda  z^5 - 4 z^4, \\
B_2 & = i a^2 \lambda  w^7 + w^6 (z (2 a \lambda ^2 (\mu -2 r-2)-2 a^2 (8 \lambda^2-3 \mu +2 r+1 ) )-2 a^2 \lambda  ) \\
	& +  w^5 z^2 (16 i a^2 \lambda  (-\bar a+\lambda ^2+\mu -1 )-i a \lambda  (20 \lambda ^2+4 \mu ^2-28 \mu +(20-8 \mu ) r+9 )\\& \quad + 8 i \lambda ^3 (r-2 \mu ) )\\
  & + w^4 (z^3 (48 a^2 \bar a-16 a^2 \lambda ^2+a (-4 \lambda ^2 (8 \mu -5)+28 \mu ^2-18 \mu +8 r (2 \lambda ^2-3 \mu +1 )+3) \\& \quad -32 \lambda ^2 \mu ^2+26 \lambda ^2 \mu +2 \lambda ^2 (8 \mu -6) r )+z^2 (-16 a^2 \lambda -8 a \lambda  (-\mu +r+2))) \\
	& + w^3 (z^4 (8 i a \lambda  (4 \bar a+2 \lambda^2-6 \mu +6 r+1 )-2 i \lambda  (\lambda ^2 (10-32 \mu )+10 \mu ^2-10 \mu \\
	& \quad + 2 r (8 \lambda ^2-2 \mu +3 )+1 ) )+z^3 (-8 i a (4 \lambda ^2-5 \mu + r+2 )-32 i \lambda ^2 \mu -8 i \lambda^2 \\ & \quad + 16 i \lambda ^2 r) + 4 i a \lambda  z^2)\\
  & + w^2 (z^5 (16 a (2 \bar a-\lambda ^2 )+\lambda^2 (20-64 \mu )+12 \mu ^2-8 \mu +16 \lambda ^2 r-8 \mu  r+4 r+1 )\\ & \quad + z^4 (16 a \lambda -2 \lambda  (-4 \mu +4 r-5))-12 a z^3 )\\
  & + w (8 i \lambda  z^6 (6 \bar a-4 \mu +2 r+1)+z^5 (8 i \mu -8 i r)-12 i \lambda  z^4 )-16 \bar a z^7+4 z^5,\\
C_2 & = 4 zw (aw^2+z^2)(a w + i \lambda z).
\end{align*}

\emph{Subcase 2.2 (a):} $\lambda \neq 0$. As $f^{(0,2)} = \lambda \ne 0$, we have $f(0,w) \ne 0$. Moreover, $f(\pm w\sqrt{-a},w) \ne 0$. Hence the restriction of the system \eqref{eq:systemCase22} onto the variety $\{z(z^2 + aw^2) = 0\}$ is homogeneous (vanishing right-hand side) and has non-zero solution. Consequently, the determinant $\Delta$ of the coefficient matrix must be divisible by $z(z^2 + aw^2)$. If we decompose $\Delta$ into homogeneous components, then
\[
	\Delta 
	=
	\det 
	\begin{pmatrix}
		A_1 & B_1 \\
		A_2 & B_2
	\end{pmatrix}
	=
	\Delta^{(7) } + \cdots + \Delta^{(10)}.
\] 
The lowest degree term $\Delta^{(7) }$ comes from the quadratic terms in $A_1$ and the quintic terms in $B_2$. Precisely,
\[
	\Delta^{(7)}
	=
	A_1^{(2)} B_2^{(5)}
	=
	16z^2(z^2 + a w^2) \left(a w^2 (i \lambda  w-3 z)+z^2 (z-3 i \lambda  w)\right).
\]
The formula for $\Delta_{\phi}$ is simpler. Indeed,
\[
	\Delta_{\phi}
	=
	\det 
	\begin{pmatrix}
		A_1 & C_1 \\
		A_2 & C_2
	\end{pmatrix}
	=
	4z(z^2 + aw^2) (A_1 (w (aw + i \lambda z)) + A_2) = 8 z^2 (z^2 + a w^2) \tilde\Delta_\phi,
\]
where $\tilde \Delta_\phi$ is a polynomial of degree $4$ in $(z,w)$. By Cramer's rule, $\phi$ agrees with $\Delta_{\phi}/\Delta$ when $\Delta \ne 0$. As $\phi$ has no singularity at the origin, $\Delta$ and $\Delta_{\phi}$ have a (nonunit) common divisor $D$, which can be supposed to be of maximal degree.  If $D$ has degree $\geq 8$, then we can argue that $f$ and $\phi$ must be quadratic and so is $g$, by  \eqref{E1}. In fact, in the reduced form, $\phi$ has a constant numerator and quadratic denominators and so $H$ is a quadratic map. 

If $D$ is of degree at most 7, then $D$ must be homogeneous of degree $7$, that is $D$ is a constant multiple of $\Delta^{(7)}$, as otherwise $\phi$ has a pole or indeterminate at the origin. Consequently, if we put $\Psi: =  \left(a w^2 (i \lambda  w-3 z)+z^2 (z-3 i \lambda  w)\right)$ and $\Phi: = (A_1(w(aw+i\lambda z)) + A_2)$, then $\Psi$ must divide $\Phi$. But by direct calculations,
\[
	\Phi = -4 z \Psi + \Phi^{(5)}
\]
where the homogeneous component of degree 5, denoted by $\Phi^{(5)}$, is
\begin{align}
 	\Phi^{(5)} & = z \bigl( 8 \lambda  z^4-2 i w z^3 \left(12 \lambda ^2-6 \mu +4 r+1\right)  -2 \lambda  w^2 z^2 (4 a-10 \mu +4 r+1) \notag \\
 	& -2 i w^3 z \left(12 a \lambda ^2-14 a \mu +12 a r+3 a+16 \lambda ^2 \mu -2 \lambda ^2-8 \lambda ^2 r\right) \notag  \\
 	& -2 a \lambda  w^4 (8 a-2 \mu +4 r+1) \bigr).
\end{align}
Hence, there are constants $\alpha, \beta \in \C$ such that
\begin{align}
\label{eq:divPsi22}
	z(\alpha z + \beta w) \Psi = \Phi ^{(5)}.
\end{align}
Equating the coefficients of $z^5$ and $z^4 w$, we have 
\[
	\alpha = -8 \lambda, \qquad \beta = 2 i (1+ 4 r - 6 \mu).
\]
Using these  and equating the coefficients of $z^3 w^2$ in \eqref{eq:divPsi22} we find that 
\begin{align}
\label{eq:lastSubcases22}
	\lambda  (4 a-4 \mu +4 r+1) = 0
\end{align}

Thus, since $\lambda \neq 0$, we obtain 
\[
r = \mu -a - \frac{1}{4},
\]
hence $a$ is real. Equating the remaining coefficient $z^2 w^3$ in \eqref{eq:divPsi22} we find that
\[
\left(a-\lambda ^2\right) (3 a+\mu ) = 0.
\]
Thus we have two subcases.

\textit{Subcase 2.2. (a) (i):} When $a = \lambda ^2$, this leads to a quadratic map. 

\textit{Subcase 2.2. (a) (ii):} Assume $\mu= -3a$. The real parameters $\lambda > 0$ and $a \leq -1/16$ are arbitrary. The others are
\[
	r = -\frac{1}{4}(16a+1), \quad 
	s= a^2 + \lambda r, \quad b = i \lambda r.
\]

\emph{Subcase 2.2 (b):} $\lambda = 0$. We want to argue that $f(\pm w \sqrt{-a}, w) \neq 0$. If we have established this fact, we can proceed similarly as in the subcase 2.2 (a) until \eqref{eq:divPsi22}. Then we equate the coefficient of $z^2 w^3$ in \eqref{eq:divPsi22} to find that 
\[
	r = \frac{1}{12} (16 \mu -3).
\]
This implies that $\mu \geq \frac{3}{16}$ as $r\geq 0$. We further determine
\[
	s = |a|^2 =\frac{1}{9} \mu^2, \quad b = 0, 
\]
and so $a = \frac{1}{3} \mu e^{i\theta}$, with $\theta\in \mathbb{R}$.

To show $f(\pm w \sqrt{-a}, w) \neq 0$ we obtain from the second holomorphic functional equation \eqref{eq:systemCase22} that 
\begin{align*}
f(\pm w \sqrt{-a}, w) = \pm\frac{\sqrt{-a}}{2} w(2 i + w(2 \mu -1)) \phi(\pm w \sqrt{-a}, w).
\end{align*}
Assume that $f(\pm w \sqrt{-a}, w) = 0$, then $\phi(\pm w \sqrt{-a}, w)=0$. Expanding into a Taylor series and using the formulas for $\phi^{(2,0)}, \phi^{(1,1)}$ and $\phi^{(0,2)}$ we have established so far, we obtain
\begin{align*}
\phi(\pm w \sqrt{-a}, w) = i \mu w +   \left(\frac 1 4 +r -\mu(1+\mu) \right) w^2 + o(2),
\end{align*}
which gives a contradiction, since $r \geq 0$. This completes the proof of \cref{lem11}.
\end{proof}

\subsection{Case II: Linear factor}

Let the map $H$ take the form
\[
	H =(f, (z- \nu w) \xi, \sigma w \xi, g), \quad \sigma > 0,\ \nu \in \mathbb{C}.
\]
The mapping equation takes the following form
\begin{align}\label{me}
	\frac{g - \bar{g}}{2i} - |f|^2 -  \left(|z- \nu w|^2 +  \sigma^2|w|^2\right)|\xi|^2 =
	Q(z,w,\bar{z},\bar{w})\left(\frac{w - \bar{w}}{2i} - |z|^2\right).
\end{align}
From \cref{lem:lindepmap} (i)(b) and \cref{lem:initalPlus} (ii) we have $\xi(0,0)=0$ and the parameters $\sigma$ and $\nu$ are related to the Taylor series coefficients of $f_3$ and $\xi$. Namely,
\[
	f_{3}^{(1,1)} = \sigma, \quad \xi^{(0,1)} = \nu = \frac{f_{3}^{(0,2)}}{2 f_{3}^{(1,1)}}, \quad \xi^{(1,0)} = 1.		
\]
We proceed similarly to Case I in \cref{sec:caseI}. Formula \eqref{e:gsol} also holds in this case. The map along the first Segre set is
\[
	f(z,0) = z, \ g(z,0) = 0, \ \xi(z,0)= \frac{z (2 \lambda  z+1)}{4 \bar{\nu}^2 z^2+1}.
\]
From this data, we can obtain
\begin{lemma}[First holomorphic functional equation] The components $f, g$ and $\xi$ satisfy
\begin{equation}\label{e:case2110}
		2(z+\nu  w)(z  g(z,w) - wf(z,w))+w^2 (\lambda  w+i z) \xi(z,w) = 0.
\end{equation}
\end{lemma}
\begin{proof} The proof is similar to that of Lemma \ref{lem1}. We omit the details.
\end{proof}
From \eqref{e:case2110}, we have that $(z+\nu  w)$ divides $(\lambda  w+i z) \xi(z,w)$. If $(z+\nu w)$ divides $\xi(z,w)$, then $\xi(z,w) = (z+ \nu w) \phi(z,w)$ for some holomorphic function $\phi$ near the origin. Thus, $H$ takes the form \eqref{ecase1} with $N= 4$, $a= -\nu^2$, $a_3= \sigma$, and $b_3 = \sigma \nu$ as in Case I. 

Suppose that $(z+\nu  w)$ divides $(\lambda  w+i z)$. Then
\begin{equation}\label{e:nulamda}
	\nu = -i \lambda.
\end{equation}
Note that $\lambda \geq 0$. The first holomorphic functional equation reduces to
\begin{equation} \label{e:case211}
	w f(z,w)-z g(z,w) -\frac{i}{2} w^2 \xi(z,w) = 0.
\end{equation}
In particular, by setting $z= 0$ we obtain
\begin{equation}\label{eq:IIfw}
	f(0,w) = \frac{i}{2} w \xi(0,w).
\end{equation}

Moreover,
\[
	\xi(z,0) = \frac{z}{1-2 \lambda  z}, \ f_w(z,0) = \frac{i z (2 \lambda  z+1)}{2-4 \lambda  z}.
\]
We proceed to determine $\xi_w(z,0)$.
\begin{lemma}
	\begin{align}
		\xi^{(2,0)} & = 4 \lambda, \\
		g^{(0,3)} & = -3i \xi^{(1,1)} + 3 f^{(1,2)}, \\
		g^{(1,2)} & = 2i \lambda,\\
		f^{(0,3)} & = \frac{3}{2} i \xi^{(0,2)}, \\
		f^{(2,1)} & = 4i \lambda,
	\end{align}
and, moreover, $g^{(2,1)} = g^{(3,0)} = f^{(3,0)}=0$, 
\end{lemma}
The coefficients $\xi^{(1,1)}$, $\xi^{(0,2)}$, and $f^{(1,2)}$ are not yet determined.

Taking derivatives of the mapping equation to order three, we obtain
\[
	\begin{pmatrix}
		1 & 0 & 0 \\
		i & -4 z & 0  \\
		0 & i & - z  \\
		0 & 0 & 1  \\
	\end{pmatrix}
	\begin{pmatrix}
		p \\
		t\\
		q
	\end{pmatrix}
	=
	\begin{pmatrix}
	 4 \lambda \xi \\
	 4z\xi\left(2 i \lambda ^2-\bar{\xi}^{(1,1)}\right)+8 i \lambda  z\\
	 \lambda -z \bar{f}^{(1,2)}-z\xi \left(\bar{\xi}^{(0,2)}-2 i \lambda  \bar{\xi}^{(1,1)}\right)\\
	 \frac{1}{3} \bar{g}^{(0,3)}+\bar{\xi}^{(0,2)}z(1+2\lambda \xi)\\
	\end{pmatrix}.
\]
We solve this system for $p, t$ and $q$ to obtain
\[
	p(z,0) = \frac{4 \lambda  z}{1-2 \lambda  z}, \quad 
	t(z,0) = \frac{-i \lambda + z(\bar{\xi}^{(1,1)} +2 i \lambda ^2 )}{1- 2 \lambda  z},
\]
and 
\[
	q(z,0) = \bar{f}^{(1,2)}+ i \bar{\xi}^{(1,1)}+\frac{z \bar{\xi}^{(0,2)}}{1-2 \lambda  z}.
\]
From this, we can solve for $\eta_w, \xi_w$, and $f_{ww}$. From the formula for $f_{ww}$, we obtain a ``reflection identity''
\[
	f^{(1,2)} = - \frac{1}{2} - 2|\sigma|^2 + 3i  \bar{\xi}^{(1,1)} + 2 \bar{f}^{(1,2)},
\]
which allows us to express real and imaginary parts of $f^{(1,2)}$ in terms of those of $\xi^{(1,1)}$ and $\sigma >0$. Namely,
\begin{align}
	\re \left(f^{(1,2)}\right) & = \frac{1}{2} + 2 \sigma^2 - 3 \im \left( \xi^{(1,1)}\right), \\
	\im \left(f^{(1,2)}\right) & = \re \left(\xi^{(1,1)}\right), 
\end{align}

The second holomorphic functional equation takes the following form
\begin{equation}\label{eq:shfeN4}
	a_1 f + b_1 \xi = c_1,
\end{equation}
where 
\begin{align}\label{e:caseii1b1}
	a_1(z,w) & = 2 (z-i \lambda  w) \left(\lambda  w^2+2 i \lambda  w+i w z+2 z\right) \\
	b_1(z,w) & = -w\biggl\{w z^2 \left(2 f^{(1,2)}-4 i \xi^{(1,1)}-4 \lambda ^2-1\right)+4 i z^2+2 \lambda  w z+2 i \lambda ^2 w^2 \notag  \\ 
	& \qquad +\lambda ^2 w^3+w^2 z \left(2 i \lambda  f^{(1,2)}+2 \lambda  \xi^{(1,1)}-i h^{(0,2)}+2 i \lambda \right)\biggr\}, \\
	c_1(z,w) & = 4 z \left(\lambda ^2 w^2+z^2\right).
\end{align}
If $(z- i \lambda w)$ divides $\xi(z,w)$ in $\mathbb{C}\{z,w\}$, we go back to the first case. Otherwise, we have $(z- i \lambda w)$ divides $b_1(z,w)$. In this case, we have
\begin{equation}\label{e:case22}
	\lambda  \left(4 \lambda  f^{(1,2)}-6 i \lambda \xi^{(1,1)}- \xi^{(0,2)}-4 \lambda ^3\right) = 0.
\end{equation}
From this, we divides into two subcases:
\subsubsection{Case 1}
If $\lambda \ne 0$, then \eqref{e:case22} implies that
\begin{align}
	\re \left(\xi^{(0,2)} \right) & = 2 \lambda \left(1- 2 \lambda ^2 + 4\sigma^2 -3 \im \left( \xi^{(1,1)}\right)\right), \\
	\im \left(\xi^{(0,2)}\right) & = -2 \lambda \re \left( \xi^{(1,1)}\right).
\end{align}
We put $\xi^{(1,1)} = u +i v$ and \eqref{eq:shfeN4} reduces to
\[
	A_1 f + B_1 \xi = C_1,
\]
with
\[
	 A_1(z,w) = 2\left(\lambda  w^2+2 i \lambda  w+i w z+2 z\right), \quad C_1(z,w) = 4 z \left(z+ i\lambda w\right), 
\]
and 
\[
	B_1(z,w) 
	=w\left(2(\lambda w -2i z) - w\left(i\lambda w - z(4 \lambda ^2-4 \sigma^2+2 i u+2 v)\right)\right)
\]

From the formula for $\xi(z,0)$ and $f_{w}(z,0)$ we obtain
\[
\xi_w(z,0) = \frac{u z+i \left(v z+\lambda  \left(z^2+4 \lambda  z-1\right)\right)}{(1-2 \lambda  z)^2}, 
\]
\begin{lemma}
	\begin{align}
		\xi^{(3,0)} & = 24 \lambda^2, \\
		\xi^{(2,1)} & = 2 \lambda\left(4 i \lambda ^2+4 u+4 i v+i\right)
	\end{align}
\end{lemma}
We have
\[
	g_{ww}(z,0) = \eta(z,0) = \frac{2 i \lambda  z}{1-2 \lambda  z}, 
\]
and therefore, $g^{(2,2)} = 8 i \lambda ^2$. Furthermore, from $g_{www}(z,0)$, we find that 
\[
	g^{(1,3)} = 3 \lambda  \left(1-4 \lambda ^2+4 \sigma^2+2 i u-6 v\right)
\]
By differentiating \eqref{eq:IIfw} twice with respect to $w$, using the formulas for $\xi(z,0)$ and $\xi_w(z,0)$ and the above conditions, we obtain
\begin{align*}
f_{ww}(z,0) = \frac{2 \lambda +(1-8 \lambda ^2+4 \sigma ^2 + 2 i u - 6 v)z - 4 \lambda  (1+ 2 \lambda ^2+2 \sigma ^2-i u-v) z^2 - 4 \lambda ^2 z^3}{2 (1-2 \lambda  z)^2},
\end{align*} 
such that,
\[
	f^{(2,2)}= 4 \lambda  \left(2\sigma^2-4 \lambda ^2+3 i u-5 v\right).
\]
From the first holomorphic equation, we also have
\[
	g^{(0,4)}=4 f^{(1,3)}-6 i \xi^{(1,2)}, \quad f^{(0,4)}= 2 i \xi^{(0,3)}
\]
The coefficients $\xi^{(1,2)}, \xi ^{(0,3)}$, and $f^{(1,3)}$ are yet to be determined.

Differentiating the mapping equation with respect to $\partial^{4-k}_{\bar z} \partial^k_{\bar w}$ at $\bar z = \bar w = w = 0$ for $k=0,\ldots, 4$ we obtain a system of linear equations for $k, l, m, n$ along the first Segre set given by
\begin{align*}
\left(\begin{array}{cccc}
1 & 0 & 0 & 0\\
1 & 6 i z & 0 & 0\\
0 & 2 & 4 i z &  0\\
0 & 0 & 3  & 2 i z  \\
0 & 0 & 0 & 4
\end{array} \right) \left(\begin{array}{c}
k\\
l\\
m\\
n
\end{array} \right) = 
\left(\begin{array}{c}
24 \lambda ^2 \xi\\
48 \lambda^2 z  -i \left(48 i \lambda^3 - 6 \bar \xi^{(2,1)} \right) z \xi\\
- 8 i \lambda^2 + 2 i \bar f^{(2,2)} z + 4 i \left(\bar \xi^{(1,2)} - i \lambda \bar \xi^{(2,1)} \right) z \xi\\
\bar g^{(1,3)} + 2 i \bar f^{(1,3)} z  + 2 i  \left(\bar \xi^{(0,3)} -3 i \lambda \bar \xi^{(1,2)} \right) z \xi\\
4 \bar f^{(1,3)}+ 6 i \bar \xi^{(1,2)} + 4 \bar \xi^{(0,3)} z  + 8 \lambda \bar \xi^{(0,3)} z  \xi
\end{array}
\right).
\end{align*}

The solvability of the system does not provide any new conditions. Using the above information, the solutions are given as follows:
\begin{align*}
k(z,0) & = \frac{24 \lambda^2 z}{1- 2 \lambda z},\\
l(z,0) & = \frac{2 \lambda  (2 i \lambda + z (i - 4 u + 4 i v))}{2 \lambda  z-1},\\
m(z,0) & = \frac{ 4 \lambda^3 -\lambda \left(1+4 \sigma ^2 - 2 i u - 6 v\right) + \bigl( 2 \lambda^2 \left(1+ 4 \sigma ^2-2 i u-6 v\right)-8 \lambda ^4 - \bar\xi^{(1,2)} \bigr)z}{2 \lambda  z-1},\\
n(z,0) & = \frac{2 \bar f^{(1,3)} + 3 i \bar \xi^{(1,2)} - 2 \left(2 \lambda  \bar f^{(1,3)} + 3 i \lambda \bar \xi^{(1,2)} -\bar\xi^{(0,3)}\right) z}{2(1-2 \lambda z)}.
\end{align*}

In order to compute $\xi_{ww}(z,0)$ we first differentiate the mapping equation with respect to $\partial_{\bar z}^{3-k} \partial_{\bar w}^k\partial_w$ for $k=0,1,2,3$ and evaluate at $\bar z = \bar w = w = 0$ to obtain the following system of linear equations for $p_w, t_w$ and $q_w$ along the first Segre set  
\begin{align*}
&\left(\begin{array}{ccc}
z & 0 & 0\\
1 & 4 i z & 0\\
0 & 4 & 4 i z \\
0 & 0 & 6
\end{array} \right)
 \left( \begin{array}{c}
p_w\\
t_w\\
q_w
\end{array}\right)
= \begin{pmatrix}
	V_1 \\
	V_2 \\
	V_3 \\
	V_4
\end{pmatrix}
\end{align*}
where
\begin{align*}
	V_1 & = 24 \lambda  z \xi_w + i \left(24 \lambda^2 \xi - k\right), \\
	V_2 & = 8 \lambda  f_w + 4 \left(2 \lambda ^2 + v + i u\right) z \xi_w + 4 i \lambda   \left(2 \left(\lambda ^2 + \sigma ^2\right) + i u + v\right) \xi + l, \\
	V_3 & = i \biggl(f_w  \left(4 \sigma ^2-2 i u-6 v+1\right)-4 i \left(2 \lambda ^4 + \sigma^2 (v + i u) + \lambda^2 \left(-4 \sigma^2 + 4 v - 1\right)\right) \xi  - i m\biggr), \\
	& \quad - 4 i \lambda \left(2 \lambda ^2-4 \sigma ^2+4 v-1\right) z h_w,  \\
	V_4 & = 3 \bar \xi^{(0,2)} \left(f_w +2 i \left(\lambda ^2+\sigma ^2\right) \xi\right) + 6 \lambda  \bar \xi^{(0,2)} z \xi_w + n.
\end{align*}

Using all the conditions obtained so far we obtain
\begin{align*}
p_w(z,0)& = \frac{4 \lambda  \left(u z+i \left(v z+\lambda  \left(z^2+4 \lambda  z-1\right)\right)\right)}{(1-2
   \lambda  z)^2},\\
t_w(z,0) & = \Bigl(\lambda  \left(1-4 \lambda^2 + 4 \sigma^2 - 2 i u - 6 v\right) + 2 \left(4 \lambda ^4-4 \lambda ^2 \sigma ^2+u^2+4 i \lambda ^2 u+v^2+8 \lambda ^2 v\right) z \\
& \qquad + 2 \lambda \left(i u + v-2 \lambda ^2\right) z^2 \Bigr)/(2 (1-2 \lambda  z)^2),\\
q_w(z,0)& = \Bigl(1+4
   v^2-2 v \left(8 \lambda ^2+4 \sigma ^2+3\right) -2 \bar \xi^{(1,2)} - 16 \lambda^4 + 16 \lambda^2 \sigma ^2+8 \lambda^2+4 \sigma^2 + 4 u^2 \\
& \qquad - 2 i \left(4 \sigma ^2+1\right) u + 4 \lambda  \bigl(\bar \xi^{(1,2)} + 8 \lambda ^2 \sigma ^2+i u \left(8 \lambda ^2-4 \sigma ^2+8 v-1\right)-8 v^2  \\
& \qquad  +v \left(3-8 \lambda^2+12 \sigma^2\right)\bigr) z+ 4 \lambda^2 \left(1-4 \lambda ^2+4 \sigma^2 + 2 i u- 2 v\right) z^2  \Bigr)/(4 (1-2 \lambda  z)^2)
\end{align*}

Considering the determinant of the augmented matrix of the above system of linear equations provides additional information: 
\begin{align*}
\xi^{(1,2)} & = (6 \lambda  \left(-8 \lambda ^4+5 \lambda ^2-8 \sigma ^4+2 u^2-2 i u \left(2 \lambda ^2-4 \sigma ^2+2 v-1\right)-2 v^2 \right.\\
& \qquad \left. + 2 \sigma ^2 \left(8 \lambda ^2+4 v-1\right)-12 \lambda ^2 v\right)-i \xi^{(0,3)})/(3 \lambda)\\
f^{(1,3)} & = (4 \xi^{(0,3)} + 3 i \lambda  \left(-48 \lambda ^4+32 \lambda ^2-64 \sigma ^4+12 u^2-2 i u \left(16 \lambda ^2-28 \sigma ^2+16 v-7\right)\right. \\
& \qquad \left. -20 v^2+4 \sigma ^2 \left(28 \lambda ^2+18 v-5\right)-80 \lambda ^2 v+6 v-1\right))/(4 \lambda)
\end{align*}

Next, we differentiate the mapping equation with respect to $\partial_{\bar z}^{2-k} \partial_{\bar w}^{k} \partial_w^2$ for $k=0,1,2$ along $\bar z = \bar w = w = 0$ to obtain a system of linear equations for $\eta_{ww}, \psi_{ww}$ and $\xi_{ww}$ along the first Segre set. Solving, we obtain:
\begin{align*}
\xi_{ww}(z,0) & = \Bigl(6 \lambda ^2 \left(2 \lambda ^2-4 \sigma ^2+i u+3
   v-1\right) + z \left(i \bar\xi^{(0,3)}-72 \lambda ^5+3 \lambda  \left(-4 \sigma ^2+2 v-1\right) \right. \\
& \qquad \left. \left(4 \sigma ^2+4 i u-4
   v+1\right)-24 \lambda ^3 \left(-5 \sigma ^2+4 v-1\right)\right) \\
 & \qquad + \lambda  z^2 \Bigl(-2 i \bar \xi^{(0,3)}+192 \lambda ^5-3 \lambda  \left(-32 \sigma ^4-12 \sigma ^2+8 u^2+4 i u \left(-8
   \sigma ^2+8 v-1\right) \right. \\
& \qquad \left. -24 v^2+48 \sigma ^2 v+4 v-1\right)+12 \lambda ^3 \left(-12 \sigma ^2-10 i u+18 v-1\right)\Bigr) \\
& \qquad +12 \lambda ^3 z^3 \left(6 \lambda ^2+2 \sigma ^2-2 i u+1\right)+12 \lambda ^4 z^4 \Bigr)/ (3 \lambda  (2 \lambda  z-1)^3 (2 \lambda  z+1)).
\end{align*}

In the next step, we obtain the third holomorphic functional equation, which is obtained by differentiation the mapping equation twice with respect to $\bar L$ along $\mathbb{H}^3$ and evaluate at $\bar w = 0$. Then we obtain:
\begin{align*}
A(z,\bar z) f(z, 2 i z \bar z) + B(z,\bar z) \xi(z,2 i z \bar z) -2 i z \bar g_{\bar w \bar w} = 0,
\end{align*}
where 
\begin{align*}
A(z,\bar z) & = 4 (i \bar f_{\bar z \bar w} + z \bar f_{\bar w \bar w}),\\
B(z,\bar z) & =  4 i z \bar \xi_{\bar w} \left(z \left(-2 \lambda +4 \lambda ^2 \bar z+4 \sigma^2 \bar z\right) - 2 \lambda  \bar z+1\right) + \bar \xi_{\bar z} \left(4 \lambda  z-8 \lambda ^2 z \bar z-8 \sigma ^2 z \bar z+4 \lambda  \bar z-2\right) \\
& \qquad + 4 z^2 \bar z (1-2 \lambda  \bar z) \bar \xi_{\bar w \bar w}+ \left(4 i z \bar z-8 i \lambda  z \bar z^2\right) \bar \xi_{\bar z \bar w}  + \left(2 \lambda  \bar z^2-\bar z\right)
\bar \xi_{\bar z \bar z}
\end{align*}

Using all the above expressions for the derivatives of $\bar f,\bar \xi$ and $\bar g$ up to order two along the first Segre set and setting $\bar z = w/(2 i z)$ we obtain the third holomorphic functional equation. It is of the form
\begin{align*}
6 \lambda z (z- i \lambda w) A_3(z,w) f(z,w) + z B_3(z,w) \xi(z,w) = 12 \lambda^2 z^2 w (z^2 + \lambda^2 w^2),  
\end{align*}
where 
\begin{align*}
A_3(z,w) & = -2 i z^2 -4 \lambda  z w -2 i \lambda^2 w^2 -4 i \lambda  z^3 + \left(8 \lambda^2 - 4 \sigma ^2 + 2 i u + 6 v - 1\right) z^2 w  \\
& \quad  + 2 \lambda   \bigl(u-i (2 \lambda^2 + 2 \sigma^2 - v +1 )\bigr) zw^2 - \lambda^2 w^3, \\
B_3(z,w) & = 12 i \lambda z^3 
+12 \lambda ^2 z^2 w
-24 \lambda  \left(2 \lambda ^2-\sigma ^2+i u+v\right) z^3 w\\
& \quad 
+ 6 i \lambda^2 \left(8 \lambda ^2-4 \sigma ^2+4 i u+4 v+3\right)  z^2 w^2 
+ 12 \lambda ^3 z w^3 
+6 i \lambda^4 w^4
-48 i \lambda ^3 z^5\\
& \quad 
+24 \lambda^2  \left(6 \lambda ^2-2 \sigma ^2+4 v-1\right) z^4 w
+ 2 i \Bigl(i \xi^{(0,3)}+48 \lambda ^5+3 \lambda  \Bigl(\left(4 \sigma ^2+1\right)^2  \\
& \qquad  \qquad  -4 i u \left(1+3 \sigma ^2-2 v\right)+8 v^2-2 \left(10 \sigma ^2+3\right) v\Bigr) - 6 \lambda ^3 \left(5 + 16 \sigma^2-16 v \right)\Bigr) z^3 w^2 \\
& \quad 
-\lambda \Bigl(2 i \xi^{(0,3)}+96 \lambda ^5 + 3 \lambda  \bigl(32 \sigma^4 + 8 \sigma^2 - 8 u^2 - 4 i u \left(1 + 10 \sigma^2 - 8 v\right) - 24 v^2 + 56 \sigma^2 v  \\
& \qquad \qquad   + 4 v - 1\bigr) - 12 \lambda ^3 \left(1 + 20 \sigma^2 - 8 i u - 16 v\right)\Bigr) z^2 w^3
+ 6 i \lambda ^3  \left(4 \lambda ^2+2 i u+1\right) z w^4 \\
& \quad
+3 \lambda^4 w^5.
\end{align*}
Hence $B_3$ must vanish along $\{z = i \lambda w\}$, which implies that
\begin{align*}
\xi^{(0,3)} & = \frac{3 i \lambda}{2}  \left(64 \lambda ^4 - 20 \lambda ^2+ 32 \sigma^4 - 4 u^2 + 8 i u \left(2 \lambda ^2 - 4 \sigma ^2+3 v - 1\right) + 20 v^2 \right. \\
& \qquad \quad  \left. - 4 \sigma ^2 \left(20 \lambda ^2+12 v-3\right) + 80 \lambda ^2 v - 8 v + 1\right). 
\end{align*}

This allows to solve in the second and third holomorphic functional equation for $f(z,w)$ and $\xi(z,w)$, giving $g(z,w)$ from the first holomorphic functional equation:
\begin{align*}
f(z,w) & =  2 \Bigl(4 z -8 z w \left(u-i \left(2 \lambda ^2-\sigma ^2+v\right)\right) + 2 \lambda  w^2 -16 \lambda ^2 z^3 \\
& \quad  - 8 i \lambda   \left(8 \lambda ^2-2 \sigma ^2+4 v-1\right) z^2 w
+ \Bigl(32 \lambda ^4-8 \lambda ^2+4 \sigma ^2+4 u^2 \\
& \qquad \qquad - 8 i u \left(2 \lambda ^2-\sigma ^2+v\right)-4 v^2+4 v \left(4 \lambda ^2+2 \sigma ^2-1\right)+1\Bigr)  z w^2\\
& \qquad 
+ i \lambda  \left(4 \lambda ^2+4 \sigma ^2+2 i u-2 v+1\right) w^3 \Bigr)/\delta(z,w),\\
\xi(z,w) & = 4 \Bigl(2 z - 2 i \lambda  w + 4 \lambda  z^2 - \left( 2 u - i \left(8 \lambda ^2-4 \sigma ^2+6 v-1\right)\right) z w \\
& \qquad + \lambda  \left(4 \lambda ^2+4 \sigma ^2+2 i u-2 v+1\right)  w^2 \Bigr)/ \delta(z,w), \\
g(z,w) & =  4 w\Bigl(2 - \left(4 u - i \left(8 \lambda ^2-4 \sigma ^2+4 v-1\right)\right) w  - 8 \lambda^2 z^2 \\
& \qquad -2 i \lambda  \left(16 \lambda ^2-4 \sigma ^2+8  v-1\right) z w  + \bigl(16 \lambda ^4+2 u^2-i u \left(8 \lambda ^2-4 \sigma ^2+4 v-1\right)\\
& \qquad \qquad -v(2 v + 8 \lambda ^2 + 4 \sigma^2 + 1)\bigr)w^2 \Bigr) /\delta(z,w),
\end{align*}
where
\begin{align*}
\delta(z,w) & = 8 -4 \left(4 u-i \left(8 \lambda^2 - 4 \sigma^2 + 4 v - 1\right)\right)  w - 32 \lambda ^2 z^2 - 32 i \lambda \left(4 \lambda ^2-\sigma ^2+2 v\right) z w \\
& \quad + 2 \bigl(32 \lambda ^4-4 \sigma ^2+4 u^2-2 i u \left(8 \lambda ^2-4 \sigma ^2+4 v-1\right) - 4 v^2+2 v \left(8 \lambda ^2 + 4 \sigma ^2+3\right)\\
& \quad \qquad -1\bigr) w^2 -16 i \lambda ^2 z^2 w + 8 \lambda  \left(6 \lambda ^2-4 \sigma ^2+i u+5 v - 1\right) z w^2 \\
& \quad + \left(u \left(8 \sigma ^2-8 v + 2\right) -i \left(48 \lambda^4 -\left(4 \sigma ^2+1\right)^2 - 8  v^2 + 2 v \left(8 \lambda ^2+12 \sigma ^2+3\right)\right)\right) w^3.
\end{align*}
These mappings are $4$-parameter families of rational mappings of at most degree three, which are holomorphic in a neighborhood of the origin.

Plugging them into the mapping equation leads to the condition
\begin{align*}
v = \frac 1 4 - \lambda^2 + \sigma^2.
\end{align*}
When plugging this into the above family of mappings, it implies that a common linear factor cancels in the numerator and denominator and the families reduce to quadratic mappings.

\subsubsection{Case 2: $\lambda = 0$} By the discussion just before \eqref{e:nulamda}, we can suppose that $\nu = 0$. Equation \eqref{eq:shfeN4} reduces to (after dividing by $z$)
\begin{equation}\label{eq:shfeN42}
	A_1 f + B_1 \xi = C_1,
\end{equation}
with
\begin{align}
	A_1 & = -2 iz(w-2 i), \\
	B_1 & = w \left(4 i z+\left(2 f^{(1,2)}-4 i \xi^{(1,1)}-1\right)zw-i w^2 \xi^{(0,2)}\right), \\
	C_1 & = -4z^2.
\end{align}
Setting $z =0$ in \eqref{eq:shfeN42} we obtain $-i w^3 \xi^{(0,2)} \xi(0,w) = 0$, which implies that either $\xi^{(0,2)} = 0$, or $\xi(0,w) = 0$. But the latter case also implies $\xi^{(0,2)} = 0$. Thus, \eqref{eq:shfeN42} simplifies to
\[
	-2i (w-2i) f(z,w) + w \left(4 i +\left(2 f^{(1,2)}-4 i \xi^{(1,1)}-1\right)w\right) \xi(z,w) = -4z.
\]
Setting $z= 0$ and substituting $f(0,w) = (i/2)w\xi(0,w)$ (from \eqref{eq:IIfw}), we obtain
\[
	2 w \xi(0,w) \left(i+w f^{(1,2)}-2 i w \xi^{(1,1)}\right) = 0,
\]
and therefore
\[
	f(0,w)= \xi(0,w) =0.
\]
Thus, we can write $\xi(z,w) = z \phi(z,w)$ for some holomorphic function $\phi(z,w)$ and $H$ takes the form \eqref{ecase1} with $N= 4$, $a=0$, $a_3 = \sigma$, and $b_3 = 0$. We are done.